\documentclass[12pt,oneside]{article}

\usepackage{amssymb}
\usepackage{amsmath}
\usepackage{amsfonts}
\usepackage{longtable}
\usepackage{tikz-cd}

\usepackage{amsfonts}
\def\dind{\mathop{\mathpalette\Ind{}}^{\text{d}} }

\newtheorem{theorem}{Theorem}[section]
\newtheorem{lemma}[theorem]{Lemma}
\newtheorem{proposition}[theorem]{Proposition}
\newtheorem{corollary}[theorem]{Corollary}
\newtheorem{definition}[theorem]{Definition}

\newtheorem{example}[theorem]{Example}

\newtheorem{fact}[theorem]{Fact}

\newtheorem{remark}[theorem]{Remark}

\newenvironment{ack}
{\begin{trivlist}  \item \textsc{Acknowledgments}~} {\end{trivlist}}

\newenvironment{proof}
{\begin{trivlist}  \item \textsc{Proof:}~} {\hfill $\Box$
\end{trivlist}}

\newenvironment{claim}
{\begin{trivlist}  \item \textsc{Claim}~} {\end{trivlist}}
\newenvironment{proof of claim}
{\begin{trivlist}  \item \textsc{Proof of Claim:}~} {\hfill $\Box$ (\textsc{Claim})
\end{trivlist}}

\newenvironment{proof of theorem}
{\begin{trivlist}  \item \textsc{Proof of Corollary \ref{sym}:}~} {\hfill $\Box$
\end{trivlist}}

\def \Th {\operatorname{Th}}

\def \res_1 {\operatorname{res_1}}

\def \res {\operatorname{res}}

\def \R{\mathcal{R}}

\def\Ind#1#2{#1\setbox0=\hbox{$#1x$}\kern\wd0\hbox to 0pt{\hss$#1\mid$\hss}
\lower.9\ht0\hbox to 0pt{\hss$#1\smile$\hss}\kern\wd0}
\def\ind{\mathop{\mathpalette\Ind{}}}
\def\Notind#1#2{#1\setbox0=\hbox{$#1x$}\kern\wd0\hbox to 0pt{\mathchardef
\nn=12854\hss$#1\nn$\kern1.4\wd0\hss}\hbox to
0pt{\hss$#1\mid$\hss}\lower.9\ht0 \hbox to
0pt{\hss$#1\smile$\hss}\kern\wd0}
\def\nind{\mathop{\mathpalette\Notind{}}}

\def\dind{\mathop{\mathpalette\Ind{}}^{\text{d}} }

\def \ra{R\langle a \rangle}

\newcommand{\ma}{\mathfrak{m}}
\newcommand{\type}{\ensuremath{\mathrm{tp}}}
\newcommand{\tp}{\ensuremath{\mathrm{tp}}}
\newcommand{\dcl}{\ensuremath{\mathrm{dcl}}}

\newcommand{\bk}{\mathbf{k}}

\begin{document}

\title{Quantifier elimination for o-minimal structures expanded by a valuational cut}
\author{Clifton Ealy and Jana Ma\v{r}\'ikov\'{a} \\Dept. of Mathematics and Philosophy, WIU}
\maketitle

\begin{abstract}
Let $R$ be an o-minimal expansion of a group in a language in which $\Th (R)$ eliminates quantifiers, and let $C$ be a predicate for a valuational cut in $R$.  We identify a  condition that implies quantifier elimination for $\Th (R,C)$ in the language of $R$ expanded by $C$ and a small number of constants, and which, in turn, is implied by $\Th (R,C)$ having quantifier elimination and being universally axiomatizable.  The condition applies for example in the case when $C$ is a convex subring of an o-minimal field $R$ and its residue field is o-minimal.
\end{abstract}

\begin{section}{Introduction}
Throughout, we let $R$ be an o-minimal expansion of a group in a language $\mathcal{L}_0$ in which $\Th (R)$ eliminates quantifiers.   We expand our language with a unary predicate $C$ used to define a valuational cut in $R$. Such a structure, $(R,C)$, is an instance of a valuational weakly o-minimal structure.  Weakly o-minimal structures were first introduced by Dickmann \cite{d}.  In \cite{womin}, Macpherson, Marker, Steinhorn  proved a basic dichotomy: Every weakly o-minimal structure is either valuational or non-valuational.  While the non-valuational ones are rather well-understood (see for example Bar-Yehuda, Hasson, Peterzil \cite{kwomin}), and have been shown to share many desirable properties with o-minimal structures, valuational weakly o-minimal structures have been, at least at a certain level of generality, less well-explored.\footnote{In \cite{cc}, Laskowski, Shaw claim to prove the existence of definable Skolem functions for structures $(R,C)$, but there is a gap in the proof.}

The main result of this paper says that if $C$ is a cut in a model $R$, and $C$ has property $\ast$ (see Definition \ref{star}), then $Th(R,C)$ eliminates quantifiers in $\mathcal{L}:=\mathcal{L}_0 \cup \{C\}$ expanded by a small number of constants.  We note that $C$ has property $\ast$ for example when $R$ is an o-minimal expansion of a field and $C$ is a convex subring of $R$ such that its residue field with structure induced from $R$ (equivalently, from $(R,C)$) is o-minimal (see Appendix A). However, it is not true, in the case when $C$ is a convex subring of $R$, that $C$ having property $\ast$ would be equivalent to the residue field of $(R,C)$ being o-minimal (see Example \ref{nonomin}).  Neither is it the case that $C$ has property $\ast$ whenever $C$ is a convex subring of $R$ (see Example \ref{nonsymforvalrings}).

Here is some context. Originally, our main interest was in the case when $R$ expands a field and $C$ is a convex subring.  A particularly well-behaved class of such structures are the $T$-convex structures, as introduced by van den Dries and Lewenberg in \cite{tconv}.  But even when one relaxes the T-convexity condition to the strictly weaker condition of having an o-minimal residue field, one maintains some desirable properties such as first-order axiomatizability (see Ma\v{r}\'{i}kov\'{a} \cite{ax}), and model-completeness of $(R,C)$ in an expansion of $\mathcal{L}$ by a small number of constants (see Ealy, 
Ma\v{r}\'{i}kov\'{a} \cite{mc}).  Later, the second author noticed (in unpublished notes) that the proof in \cite{mc} can rather easily be extended to yield quantifier elimination.  The question remained whether the assumption of o-minimality of the residue field is necessary.  We show here that it can be replaced by the more general condition $\ast$.

In \cite{mc}, to prove model completeness, one first establishes Theorem 3.3, p. 244, a criterion for elementary extensions of $(R,C)$.  While \cite{mc} assumes that $R$ expands a field and $C$ is a convex subring, in this paper we relax these conditions to $R$ expands a group and $C$ defines a valuational cut.  The remainder of the proof in \cite{mc} relies heavily on the assumption of o-minimality of the residue field and is here replaced by an entirely different approach. We also provide a much simplified proof of (an analog of) Theorem 3.3, p.244 \cite{mc}, and correct some errors from \cite{mc}.  

In particular, Fact 1.6, p.236 \cite{mc} is misstated: it holds for cuts, but not for non-cuts; in \cite{mc}, we only apply it to cuts.  More seriously, we claim to have proven that what we call property $\ast$ in this paper always holds (Theorem 2.7, p.239, \cite{mc}).  It does hold where we use it to establish model completeness (see Appendix), but it is not always true, as the examples in the next section show.  In total, Corollary 2.9, Lemma 3.1, Lemma 3.2, and Theorem 3.3 of \cite{mc} require property $\ast$, and Corollary 3.4 requires the hypothesis of either property $\ast$ or power-boundedness (see Appendix).

\bigskip\noindent
Our notation and set-up are as follows. We let $\mathbb{N}=\{ 0,1,\dots \}$. The variables $i,j,k,l,m,n, \dots$ range over $\mathbb{N}$, and, more generally, $\alpha, \beta, \gamma , \dots $ denote ordinals. 

Definable shall mean definable with parameters, unless indicated otherwise. If $M$ is a structure, then sometimes we write {\em $M$-definable\/} to mean definable in the structure $M$ with parameters from $M$.  Let $M$ be an o-minimal structure. If $M \preceq \mathcal{M}$ and $a\in \mathcal{M} \setminus M$, then we denote by $M\langle a \rangle$ the structure generated by $a$ in $\mathcal{M}$ over $M$ (note that then $M \preceq M
\langle a \rangle \preceq \mathcal{M}$).

We follow Marker, Steinhorn \cite{ms} (and many others) in saying that a type $q\in S_1 (M)$ is a {\em cut\/} if there are nonempty disjoint subsets $C^- $ and $C^+$ of $M$ such that $C^- <C^+$, $C^- \cup C^+ = M$, $C^-$ has no supremum in $R$, and $q$ consists of the formulas $c<x$ for all $c\in C^-$ and $x<c$ for all $c\in C^+$.  (Thus, cuts are precisely the nondefinable 1-types.)   We call a cut in $M$
{\em valuational\/} if there is $\epsilon \in M^{>0}$ such that $C^-$ is closed under addition by $\epsilon$.

As stated above, we let $R$ be an o-minimal expansion of a group in a language $\mathcal{L}_0$.  We assume that $\mathcal{L}_0$ is such that $T_0 =\Th (R)$ eliminates quantifiers in $\mathcal{L}_0$, and that $\mathcal{L}_0$ and $T_0$ have been expanded by definitions as follows: for each $\mathcal{L}_0$-formula $\phi(x_1 , \dots , x_n ,y)$ such that $$T_0 \vdash \forall x_1 \dots \forall x_n \; \exists ! y\; \phi (x_1 , \dots ,x_n , y)$$ we add a new function symbol $f$ to $\mathcal{L}_0$ and the axiom $$\phi (x_1 , \dots , x_n , f(x_1 , \dots , x_n ))$$ to the theory $T_0$.  Then $T_0$ is universally axiomatizable.

By $C$ we shall denote a unary predicate for a convex, downward closed set such that 
the type $$\{x>r\colon r\in C\}\cup \{x<r\colon r>C\}$$ is a cut. 
 We write $C(M)$ for the realization of $C$ in a structure $M$. We consider $(R,C)$ in the language $\mathcal{L}= \mathcal{L}_0 \cup \{C\}$, and we use $T$ to denote the theory of $(R,C)$ (in the language $\mathcal{L}$).  
 We work inside a monster model $(\mathcal{R},C)$ of $T$.  In particular, if $M \preceq \mathcal{R}$, then in $(M,C)$ the predicate $C$ is realized by the set $C(\mathcal{R})\cap M$.

Finally, for $A\subseteq \mathcal{R}$, by $S_1(A)$ we mean the 1-types in the (o-minimal!) language $\mathcal{L}_0$ over $A$.

We will repeatedly use the following properties of nondividing which hold in all theories (see \cite{adler} Lemma 5.2 for a proof):

\begin{itemize}
 \item (monotonicity) If $A\dind_C B$, $A_0 \subseteq A$ and $B_0 \subseteq B$, then $A_0\dind_C B_0$.

\item (base monotonicity)  Suppose $D \subseteq C \subseteq B$. If $A \dind_D B$, then $A \dind_C B$.

\item (left transitivity) Suppose $D \subseteq C \subseteq B$. If $B \dind_C A$ and $C \dind_D A$, then $B \dind_D A$.
\end{itemize}

\noindent Furthermore, we will use that forking and dividing are the same in o-minimal theories. (See the remarks preceding Proposition 2.8 together with Corollary 5.6 in \cite{CS} for a proof that forking is the same as dividing in weakly o-minimal theories).  This allows us to also use the following:
\begin{itemize}
    \item (extension) If $A\ind_C B$ and $\hat B\supseteq B$, then there is $A'\equiv_{BC} A$ such that $A'\ind_C \hat B$. 
\end{itemize}

\noindent Finally, we note a useful, but sometimes overlooked, observation about o-minimal structures.

\begin{fact}\label{decreasing}
Let $p\in S_1(R)$ be non-isolated.  Then there is no $\mathcal{L}_0$-definable over $R$ decreasing function mapping the set of realizations of $p$ in $\mathcal{R}$ to itself.
\end{fact}
\begin{proof}
A decreasing function with the above property would have a fixed point $a$ realizing $p$.  But then $a \in \dcl(R)$, a contradiction with $p$ including $\{x\neq r| r\in R\}$.
\end{proof}

\begin{ack} 
This work was supported by a grant from the Simons Foundation (318364, J.M.).
The second author would like to thank the program committee of the ASL 2019 Winter Meeting for motivating her to settle the questions answered in this paper. 
Both authors thank the members of the Kurt G\"odel Research Center for their hospitality during a visit in 2019/20.
\end{ack}

\end{section}

\begin{section}{The (non)-symmetry of forking in Morley sequences in invariant one-types in o-minimal structures}

In \cite{mc}, it is claimed that, in o-minimal expansions of groups, forking in Morley sequences in invariant one-types is symmetric (p. 239, Theorem 2.7).  This result is used in the proof of Theorem 3.3, p. 244, \cite{mc}, whose analog we wish to establish here.
We refer the reader to \cite{mc} for the necessary definitions.

Unfortunately, the proof of Theorem 2.7, \cite{mc}, has a gap.  In fact, the theorem is false in the stated generality.  In the current paper, cuts for which Theorem 2.7 of \cite{mc} holds are said to satisfy $\ast$  (see Definition \ref{star}).  Contra Theorem 2.7, these are not all cuts, as follows from Proposition \ref{counterexamplesym}.  Nor are they all valuational cuts (Example \ref{counterexamplesym_valuational}) or even all cuts given by valuation rings (Example \ref{nonsymforvalrings}). 
While this does not pose a problem for the model completeness result in \cite{mc} (Theorem 2.7 does hold in the context of \cite{mc}, namely, when $V$ is a convex subring of an o-minimal field $R$ such that the corresponding residue field is o-minimal -- see Appendix), here we need to bridge the gap in the proof of Theorem 3.3. We do this by additionally assuming that the cut $C$ has property $\ast$.

\begin{definition}
  Given an $A$-invariant type $p$ over $\mathcal R$, and $B\supseteq A$,  a { 
\em Morley sequence in $p$ over $B$\/} is any sequence $t_1 , t_2 , \dots $ in $\mathcal{R}$ 
constructed as follows: let $t_1\models p|_B$ and having defined $t_1,\dots,t_n$, let $t_{n+1}\models p|_{Bt_1 \dots t_n }$.  
\end{definition}

\noindent Note that in a Morley sequence over $B$, $t_i \ind_B t_1\dots t_{i-1}$ and moreover the sequence is indiscernible over B.

\begin{proposition}\label{equivalent to star}
 Let $p \in S_1 (\mathcal{
\R})$ be the global type $C(R)<x$ and $x<\mathcal{R}^{>C(R)}$, and let $t_1 , \dots ,t_n$ form a finite Morley sequence in $p$ over $R$. Let $q \in S_1 (\mathcal{
\R})$ be the global type implied by $r<x$ for all $r\in \mathcal{R}$ with $r<R^{>C(
R)}$, and  $x<R^{>C(R)}$ (i.e. the $R$-invariant type at the opposite side of the cut from $p$).  Then the following are equivalent: 
 \begin{enumerate}
     \item There is no function $f$,  $\mathcal{L}_0$-definable over $Rt_2 t_3 \dots t_{n-1}$, which maps a cofinal sequence in $C(R)$ onto a coinitial sequence in the convex hull of $R^{>C(R)}$ in $\mathcal{R}$. 
     \item $t_1\ind_R t_2 \dots t_n$.
     \item $t_n, \dots, t_1$ is a finite Morley sequence in $q\in S_1(\mathcal{R})$ over $R$.
 \end{enumerate}
\end{proposition}
\begin{proof}
First note that for $n=2$, (1) is always true by Fact \ref{decreasing}, as an $\mathcal{L}_0$-definable over $R$ function mapping  a cofinal segment of $C(R)$ to a coinitial segment of the convex hull of $R^{>C(R)}$ in $\mathcal{R}$ would be a decreasing $\mathcal{L}_0$-definable over $R$ function mapping a complete type over $R$ to itself.  Also, it is always the case that $t_1\ind_R t_2$ and $t_1\models q|_{Rt_2}$, since if either of these failed to be true, there would be some $\mathcal{L}_0$-definable over $R$ function $f$ with $f(t_2)>t_1$.  As already noted, such a function cannot be decreasing.  But nor can it be increasing, as $f^{-1}(t_1)$ would be less than $t_2$, contradicting the choice of $t_2$.

So we will assume $n>2$, and set $\underline{t}=t_2,\dots, t_{n-1}$.  

 \begin{trivlist}
 
 \item[$3. \implies 2.$] Clearly, $t_1\models q|_{R\underline{t}t_n}$ implies $t_1 \ind_R \underline{t}t_n$.
 
  \item[$2.\implies 1.$] Assume that $f_{\underline{t}}$ maps a cofinal sequence in $C(R)$ to a coinitial sequence of the convex hull of $R^{>C(R)}$ in $\mathcal{R}$.  Then $t_1<f_{\underline{t}}(t_n)<R^{>C(R)}$, and thus $t_1$ is contained in the interval $(t_2, f_{\underline{t}}(t_n))$.  This witnesses $t_1\nind_R \underline{t}t_n$.

  \item[$1. \implies 3.$] Assume that $t_1\not\models q|_{R\underline{t}t_n}$, and, without loss of generality, assume that $n$ is the least such.  Then there is an element  $s\in R\langle\underline{t}t_n\rangle$ realizing $q|_R$ with $t_1<s$.  Write $s=f_{\underline{t}}(t_n)$ for some  $\mathcal{L}_0$-definable over $R$ function $f$.  We claim that $f_{\underline{t}}$ maps a cofinal sequence in $C(R)$ onto a coinitial sequence in in the convex hull of $R^{>C(R)}$ in $\mathcal{R}$.
  
  We first show that $f_{\underline{t}}$ is strictly decreasing on some interval $I$ containing a cofinal sequence in $C(R)$ and that $f_{\underline{t}}(I\cap C(R))$ is contained in the convex hull of $R^{>C(R)}$ in $\mathcal{R}$:
  By o-minimality of $R\langle \underline{t} \rangle$, $f_{\underline{t}}$ is continuous and strictly monotone on some interval $I$ with endpoints in $\dcl{(R \underline{t})}$ and containing $t_n$.  Since $\underline{t} , t_n$ is a Morley sequence in $p$ over $R$, the left endpoint of $I$ is contained in the convex hull of $C(R)$ in $R\langle \underline{t}\rangle$. 
  For sufficiently big $c\in C^{-}$ we have $t_1< f_{\underline{t}}(c)$, since $t_n \in \{x\colon f_{\underline{t}}(x)>t_1  \}$, and $\{x\colon f_{\underline{t}}(x)>t_1  \}$ is $\mathcal{L}_0$-definable over $R\langle t_1 \underline{t}\rangle$.  Moreover, $f_{\underline{t}}(c)<R^{>C(R)}$ would yield a contradiction with the minimality of $n$.  So we have $f_{\underline{t}}(t_n )<f_{\underline{t}}(c)$, hence $f_{\underline{t}}$ is decreasing on $I$, and $f_{\underline{t}}(I \cap C(R))$ is contained in the convex hull of $R^{>C(R)}$ in $\mathcal{R}$.

  It remains to show that $f_{\underline{t}}(C(R) \cap I)$ is coinitial in the convex hull of $R^{>C(R)}$ in $\mathcal{R}$.  Suppose not.  Then there is $c\in R^{>C(R)}$ with $c<f_{\underline{t}}(C(R) \cap I)$.  Then $f_{\underline{t}}^{-1} (c)< t_n$ but realizes $p|_R$.  This contradicts the definition of $t_n$.
  
 \end{trivlist}
\end{proof}

\begin{definition}\label{star} Let $p \in S_1 (\mathcal{
\R})$ be the global type $C(R)<x$ and $x<\mathcal{R}^{>C(R)}$. We say that the cut defined by $C$ in $R$ has property $\ast$ (or (R,C) has property $\ast$) if whenever $\underline{t}=t_1 , \dots ,t_n$ form a finite Morley sequence in $p|_{R}$, then the equivalent conditions of Proposition \ref{equivalent to star} hold.
\end{definition}

\noindent We use $C$ both for a predicate in the expanded language $\mathcal{L}$ and to denote a cut, i.e. a non-definable 1-type in the the o-minimal language $\mathcal{L}_0$.  Property $\ast$ is better thought of as a property of an o-minimal 1-type rather than as a property of the predicate $C$.

\begin{proposition}\label{sym} Let $p\in S_1 (\R)$ be a global $R$-invariant type such that $p|_R$ is the type corresponding to $C(R)$ and assume further that $(R,C)$ has property $\ast$. Let $t_1 , \dots ,t_n$ be a finite Morley sequence in $p$ over $R$.  Then 
$$t_1 \dots t_{k} \ind_R t_{k+1} \dots t_n \mbox{ and } t_{k+1} \dots t_n \ind_R t_1 \dots t_k .$$
for all $k$ with $1\leq k<n$.
\end{proposition}

\begin{proof}
Use 2. of Proposition \ref{equivalent to star} and the fact that for any $a,b_1, b_2, c$, one has $a\ind_c b_1 b_2 \implies a\ind_{cb_1} b_2$, and $a\ind_{cb_1} b_2 \textrm{ and } b_1 \ind_c b_2 \implies ab_1\ind_c b_2$.   
\end{proof}

\begin{proposition}\label{counterexamplesym} 
$(R,C)$ has property $\ast$ implies that $C$ defines a valuational cut in R.
\end{proposition}

\begin{proof}
Consider $(R,C)$ with $C$ non-valuational.  We denote by $p\in S_1(\mathcal{R})$ the type implied by $C(R)<x$ and $x<\mathcal{R}^{>C(R)}$. Let $t_1 , t_2 , t_3$ be a finite Morley sequence in $p$ over $R$. In particular, $t_3 < t_2 < t_1$. 

Since $C$ is non-valuational, $0<t_1-t_2<R^{>0}$, so $t_2-(t_1-t_2) \models p|_R$ and thus $t_3<t_2-(t_1-t_2)$.  That is, $t_1 <t_2 + (t_2 - t_3)$.  Likewise, $0<t_2-t_3<R^{>0}$, so $t_2+(t_2-t_3)\models p|_R$. Hence the formula $t_2< x < t_2+(t_2-t_3)$  witnesses $t_1 \nind_R t_2t_3$.
\end{proof}

The remainder of this section is devoted to the construction of several examples.  First an example of a structure $(R,C)$ where the cut defined by $C$ does not have property $\ast$.

\begin{example}\label{counterexamplesym_valuational}
Let $R$ be the field of the real algebraic numbers and let $R'=R\langle\epsilon\rangle$, the structure generated over the real algebraic numbers by a positive infinitesimal $\epsilon$.   Let $C$ be the cut of $\pi$. The cut of $\pi$ is a non-valuational cut in $R$ but a valuational cut in $R'$.  However, the same argument as in Proposition \ref{counterexamplesym} shows $t_1\nind_{R'}t_2t_3$.
\end{example}
 By Theorem 6.3, p. 5470 \cite{womin}, there is a canonical way in which $(R,C)$ defines a valuation ring.  In Example \ref{counterexamplesym_valuational}, this valuation ring is the convex hull of $\mathbb{Q}$ in $R'$, so its residue field is o-minimal by Corollary 2.4, p. 244 \cite{mc}.  It is therefore not possible to weaken the assumption in Proposition \ref{Sigma implies star}, which says that if $C$ is a cut corresponding to a convex subring of $R$ such that the associated residue field is o-minimal, then $C$ has property $\ast$, 
to: ``the residue field of the canonical valuation associated to $(R,C)$ is o-minimal".  

On the other hand, if $C=V$ is a convex subring of $R$, then the o-minimality of the residue field is not a necessary condition for
$C$ to have property  $\ast$, as the following example demonstrates.

\begin{example}\label{nonomin}
Let $R$ be a big elementary extension of the real exponential field, and let $a\in R^{>\mathbb{Q}}$.  We set $V=\bigcup_n (-a^n , a^n )$. The residue field of $(R,V)$ is not o-minimal, since the inverse function of the exponential induces a map $\ln \colon \bk^{>0} \to \bk$, where $\ln (\bk^{>0}) \subseteq \bk$ is bounded above but does not have a supremum in $\bk$.  However, letting $\mathcal{O}$ be the convex hull of $\mathbb{Q}$, we see that $(R, V)$ and $(R, \mathcal{O})$ are interdefinable via the $\mathcal{L}_0$-definable function $f(x)=\frac{\ln{x}}{\ln{a}}$.  The residue field of $(R, \mathcal{O})$ is $\mathbb{R}$ and is thus o-minimal. (In fact, $(R, \mathcal{O})$ is even $T$-convex.)  Thus, the cut corresponding to $\mathcal{O}$ has property $\ast$, and therefore so does the cut corresponding to  $V$.
 
\end{example}
While Example \ref{counterexamplesym_valuational} shows that $(R,C)$ might not satisfy $\ast$, one  might still hope that $\ast$ holds whenever $C$ is a predicate for a valuation ring, but this is also not true.  An example illustrating this is a bit more difficult to construct, and we shall devote the rest of the section doing so.

\begin{example}\label{nonsymforvalrings}
Set $R_0 = \mathbb{R}_{\exp}$, the real exponential field, and let $\mathcal{M}$ be a countably saturated elementary extension of $R_0$. We pick $a_1 \in \mathcal{M}^{>\mathcal{O}}$ and set $R_1 = R_0 \langle a_1 \rangle$ and $V_1 = \bigcup_n (-a_{1}^n , a_{1}^n ) \subseteq R_1$.  We define $(R_{k+1}, V_{k+1})$ inductively by letting $a_{k+1} \in \mathcal{M}$ be such that $V_k < a_{k+1} < R_{k}^{>V_k}$, $R_{k+1}=R_k \langle a_{k+1} \rangle$, and $V_{k+1} = \bigcup_n (-a_{k+1}^n , a_{k+1}^n ) \subseteq R_{k+1}$.  Finally we set $R= \bigcup_k R_k$ and $V=\bigcup_k V_k$.  Our aim is to show that $V$ does not have property $\ast$.

 \begin{lemma}\label{exlemma} Let $f(x,y)=y^{\frac{\ln y}{\ln x}}$, and 
  let $t_3 < t_2 < t_1$ be a finite Morley sequence in $p$ over $R$, where $p\in S_1 (\mathcal{R})$ is the type $V<x<\mathcal{R}^{>V}$.  Then 
  $$t_1 <f(t_3, t_2 )< R^{>V}.$$ 
 \end{lemma}
\begin{proof}
 For each $i$, we let $g_i \in R_{i+1}$ be such that $a_{i}^{g_i }=a_{i+1}$, i.e. $g_i = \frac{\ln{a_{i+1}}}{\ln{a_{i}}}$. 
 
 \begin{claim}$1:$
 $\mathcal{O}<g_i <R_{i}^{>\mathcal{O}}$, where $\mathcal{O}$ is the convex hull of $\mathbb{Q}$ in $R_i$.
\end{claim}
 
\noindent
By our choice of $a_{i+1}$, $g_i>k$ for all $k\in\mathbb{N}$.  
 Moreover, if there was $r\in R_{i}^{>\mathcal{O}}$ with $r<g_i$, then $V_i < a_{i}^{r}<a_{i}^{g_i}=a_{i+1}$.  But since $a_{i}^{r} \in R_i$, this would imply that $a_{i}^{r} \in V_i$, a contradiction.

 \begin{claim}$2:$
 $g_{i+1}<g_{i}^{\frac{1}{k}}$ for all $k\in \mathbb{N}$. 
 \end{claim}
 
 \noindent
 By Claim 1, $g_{i+1}< g_i$.  For any $k$, $f(x)=x^k$ is increasing and $f(\mathcal{O})\subseteq \mathcal{O}$, so $g_{i+1}^k \geq g_i$ is impossible.

\medskip\noindent

Note that 
 $\{ a_{i+1}^{g_i}\}_{i=1}^{\infty}$ is a coinitial sequence in $R^{>V}$: If not, then there would be $r\in R^{>V}$ such that $r<a_{i+1}^{g_i}$ for all $i$.  But $r\in R_{j}^{>V_j}$ for some $j$, hence $a_{j+1}^{g_j}<r$, a contradiction.  So Claim 3 will complete the proof.

\begin{claim}$3:$ 
For $i\in \{ 1,2 \}$, we set $h_i := \frac{\ln{t_i}}{\ln{t_{i+1}}}$ (so $t_i = t_{i+1}^{h_i}$).
Then, for all $i$, $$t_1 < t_{2}^{h_2}<a_{i+1}^{g_i}=f(a_i , a_{i+1}).$$
\end{claim}

\noindent
To prove that $t_{2}^{h_2}<a_{i+1}^{g_{i}}$ for all $i$, assume towards a contradiction that $i$ is such that $t_{2}^{h_2}\geq a_{i+1}^{g_i}$.  Then $$h_2 \geq \frac{g_i \ln{a_{i+1}}}{\ln{t_2}}\geq \frac{g_i \ln{a_{i+1}}}{g_{i+1}\ln{a_{i+2}}}=\frac{g_i }{(g_{i+1})^2}.$$
By Claim 2, we have $\frac{g_i }{(g_{i+1})^2}>g_{i}^{\frac{1}{2}}$, so $h_2 > g_{i}^{\frac{1}{2}}$.  It follows that \[t_2 =t_{3}^{h_2}>a_{i+1}^{(g_{i}^{\frac{1}{2}})}\in R_{i+1}^{>V},\] a contradiction with $t_2$ realizing $p|_R$. 

It is left to show that $t_1 <t_{2}^{h_2}$.  If not, then $t_{2}^{h_1}>t_{2}^{h_2}$, so $h_1 > h_2$.  We would then be able to find $k$ such that 
$t_{2}^{\frac{1}{h_1}}<a_k$, which is equivalent to $t_{2}^{\frac{\ln{t_2}}{\ln{a_k}}}<t_1$.  On the other hand, $\frac{\ln{t_2}}{\ln{a_k}}> \frac{\ln{a_{k+1}}}{\ln{a_k }}=g_k$.  So \[t_{2}^{\frac{\ln{t_2}}{\ln{a_k}}}>t_{2}^{g_k}>a_{k+1}^{g_k}>V,\] a contradiction with $t_{2}^{\frac{\ln{t_2}}{\ln{a_k}}}<t_1$.
\end{proof}
\medskip\noindent
Lemma \ref{exlemma} shows that $t_1 \nind_R t_2 t_3$, and hence $V$ does not have property $\ast$.

\end{example}

\smallskip\noindent 
\end{section}

We conclude this section with a lemma showing that property $\ast$ is preserved by passing to certain superstructures.

\begin{lemma}\label{superstructures}
Suppose $C$ defines a cut in $R$ which has property $\ast$, and let $p\in S(\mathcal{R})$ be the invariant type implied by \[\{x>c |c\in C(R) \} \cup \{x<r| r\in 
\mathcal{R} \textrm{ and } r>C(R)\}.\] 
and $q\in S(\mathcal{R})$  the type implied by \[\{x<r |r\in R^{>C(R)} \} \cup \{x>r| r\in \mathcal{R} \textrm{ and } r<R^{>C(R)}\}.\]  
Then any of the following conditions implies $C$ defines a cut in $R\langle (a_i)_{i<\kappa}\rangle$ which has property $\ast$:
\begin{enumerate}
    \item $(a_i)_{i<\kappa}$ forms a Morley sequence in $p$ over $R$ and \[\dcl((a_i)_{i<\kappa}\cap p|_R)>C(\mathcal{R}).\]
    \item $(a_i)_{i<\kappa}$ forms a Morley sequence in $q$ over $R$ and \[dcl_R ((a_i)_{i<\kappa}\cap p|_R) \subseteq C(\mathcal{R}).\]
    \item  $\dcl_R ((a_i)_{i<\kappa}\cap p|_R)=\emptyset$.
\end{enumerate}
\end{lemma}
\begin{proof}
Assume (1) holds.  Take a Morley sequence $(t_i)_{i<\omega}$ in $p$ over $R\langle (a_i)_{i<\kappa}\rangle$.   
Then $(a_i)_{i<\kappa} (t_i)_{i<\omega}$ is also a Morley sequence in $p$ over $R$.  Note that $t_1 (a_i)_{i<\kappa}\ind_{R} t_2 \dots t_n $, because if not, there would be a finite subtuple $a_{i_0}, \dots, a_{i_k}$ with \[ a_{i_0} \dots a_{i_k} t_1 \nind_{R} t_2 \dots t_n , \] and, by indiscernibility of the Morley sequence, we would have \[t_1 t_2 \dots t_{k+2}\nind_{R} t_{k+3} \dots t_{k+n+1},\] a contradiction with $C$ having property $\ast$. 
From $t_1 (a_i)_{i<\kappa}\ind_{R} t_2 \dots t_n $, it follows that $t_1 \ind_{R\langle (a_i)_{i<\kappa}\rangle} t_2 \dots t_n$.

If (2) holds, one carries out the analogous argument with the Morley sequence $(t_i)_{i<\omega}$ in $q$.

Now assume (3).  Suppose for a contradiction that $C$ does not define a cut in $R\langle (a_i)_{i<\kappa} \rangle$ that has property $\ast$, and that this is witnessed by a finite Morley sequence $t_k < t_{k-1}< \dots < t_1 $ in $p|_{R\langle (a_i)_{i<\kappa} \rangle}$ and an $R$-definable function $f$, i.e. $t_1 < f_{\underline{a}} (\underline{t})$, where $\underline{a}$ is a finite subsequence of $(a_i)_{i<\kappa}$, $\underline{t} = t_2 , \dots , t_k$ and $f_{\underline{a}}(\underline{t}) \models p|_{R\langle (a_i)_{i<\kappa} \rangle}$.  Then the cut defined by $C$ in $R\langle \underline{a} \rangle $ does not have property $\ast$.  Thus, if we could prove that $C$ defining a cut in $R$ with property $\ast$ implies that $C$ also defines a cut with property $\ast$ in  $R\langle a\rangle$, where $a$ is a singleton and $\dcl_R (a)\cap p|_R=\emptyset$, then by applying this fact finitely many times, we could show $C$ defines a cut with property $\ast$ in  $R\langle \underline{a}\rangle$.  So towards a contradiction, we assume $C$ does not define a cut with property $\ast$ in $R\langle a\rangle$.

Let $t_k < t_{k-1}< \dots < t_1 $ be a finite Morley sequence in $p$ over $R\langle a \rangle$ and $f$ an $R$-definable function with $t_1 < f_a (\underline{t})$, where $\underline{t} = t_2 , \dots , t_k$, and $f_{a}(\underline{t}) \models p|_{R\langle a \rangle}$.

Set $t'_1 = f_a (\underline{t})$ and let $t'_k < \dots <t'_2<t'_1$ be a finite Morley sequence in $p$ over $R\langle a \rangle$ so that $t'_2 < t_k$.  Since $\type(\underline{t}/R\langle a \rangle)=\type(\underline{t}'/R\langle a \rangle )$, we have \[t'_1 = f_a (\underline{t})<f_a (\underline{t}')<R\langle a \rangle^{>C(R\langle a \rangle )}.\]

We now set $t''_1=f_a (\underline{t}')$.
\begin{claim}
The sequence $t''_1 , t'_2 , \dots , t'_k$ is a finite Morley sequence in $p$ over $R$.
\end{claim}
\begin{proof}
This is proved by induction on $k$. We clearly have $t''_1 \models p|_{R}$.  Now suppose that $t''_1 , t'_2 , \dots ,t'_l$ is a finite Morley sequence in $p$ over $R$.  If $t''_1 , t'_2 , \dots t'_{l+1}$ was not, then we would have $C<h(t''_1 , t'_2 , \dots t'_l)<t'_{l+1}$.  Let $H(x)=h(x, t'_2 , \dots ,t'_l )$ and consider $H(t'_1)$.  The function $H$ is continuous and strictly monotone on some interval, $I$, with left endpoint in $\dcl(t'_2 \dots t'_l)$ and realizing the type $p|_{R}$, and right endpoint in the convex hull of $R^{>C(R)}$ in $\mathcal{R}$.  

If $H$ is increasing on that interval, then $H(t'_1)<r$ for some $r\in C(R)$, since otherwise $t'_1, \dots, t'_{l+1}$ would not be a Morley sequence in $p$ over $R$.  But if $H(t'_1)<r$, then any $r'>r$, $r'\in C(R)$ would have to be such that $t'_1<H^{-1}(r')<t''_1$, a contradicting that $C$ defines a cut with property $\ast$ in $R$.

If $H$ is decreasing, then a coinitial segment of $R^{>C(R)}$ would have to be mapped to a cofinal segment of $C(R)$, since if there were any $r \in R^{>C(R)}$ with $H(r)\models p|_R$ this would necessarily be less than $t'_{l+1}$ (since $H(r)<H(t''_{1})$), contradicting $t'_1, \dots, t'_{l+1}$ being a Morley sequence in $p$ over $R$.  If, on the other hand, $r\in C(R)$ is such that $H^{-1}(I \cap R^{>C(R)})<r$, then $H^{-1}(r)>t''_1$, hence $H^{-1}(r)$ cannot satisfy $p|_R$ because the cut defined by $C$ in $R$ has property $\ast$.  But $H$ mapping a coinitial segment of $R^{>C}$  to a cofinal segment of $C(R)$ contradicts $C$ defining a cut in $R$ with property $\ast$. 
\end{proof}
Note that 
$a \in \dcl_R (t''_1 , \underline{t}')$.  Extend $t''_1 , \underline{t}'$ to a finite Morley sequence $t''_1 , \underline{t}', \underline{t}''$ in $p|_{R}$. Then $t''_1 < f_a (\underline{t}'')$ and $f_a (\underline{t}'')\models p|_{R\langle a \rangle}$.  But $f_a (\underline{t}'')=f_{g(t''_1 , \underline{t}')}(\underline{t}'')$ for some $R$-definable function $g$.  Hence the function$f(g (\dots )\dots )$ and the finite Morley sequence $t''_1 , \underline{t}' , \underline{t}''$ in $p|_{R}$ witness the failure of $\ast$ in the cut defined by $C$ in  $R$, a contradiction.

\end{proof}

\begin{section}{Elementary extensions}

The theorem below is an analog of Theorem 3.3, p. 244 in \cite{mc}. The difference between the two statements is that in \cite{mc}, $R$ is assumed to expand a field and $C$ (called ``$V$'' in \cite{mc}) is assumed to be a proper convex subring of $R$.  Moreover, Theorem 3.3 in \cite{mc} needs an additional assumption such as $V$ has property $\ast$  (though $(R,V)\models \Sigma$, i.e. the induced structure on the residue field of $(R,V)$ is o-minimal, would also suffice) to be correct.

Note that while Theorem 3.3 in \cite{mc} is an implication, we state it here as an equivalence. The easy direction (which went unstated in \cite{mc}) being if $(R,C)\preceq (R\langle a \rangle , C)$, then we cannot have $C(R)<a,f(a)<R^{>C(R)}$, where $f$ is an $R$-definable function and $a\in C(R\langle a \rangle )$ and $f(a)>C(R\langle a \rangle)$.  Else, \[(R\langle a \rangle , C)\models \exists x \in C\; f(x)>C,\] hence \[(R,C) \models \exists x\in C \; f(x)>C, \] yielding a contradiction with $f$ being increasing (see Fact \ref{decreasing}).

\begin{theorem}\label{elextnew}
Suppose $C$ is a cut in $R$ with property $\ast$, and let $a\in \mathcal{R}\setminus R$.  Then $(R,C)$ fails to be an elementary submodel of  $(\ra, C)$ if and only if there is an element of $C(\ra )$ greater than any element of $C(R)$ (WLOG, we may assume this element is $a$) and an $R$-definable function, $f$, such that \[C(\ra) <f(a)<R^{>C(R)}.\] 
\end{theorem}
\begin{proof}
 This may be proved in an almost identical fashion to the proof of Theorem 3.3 \cite{mc} (which encompasses paragraphs 2 and 3), and below, for the reader familiar with that proof, we outline the minor changes needed.
 \begin{enumerate}
 \item All instances of ``$V$'' need to be replaced by ``$C$''. 
 \item Replace Theorem 2.7, p. 239 \cite{mc} with Corollary \ref{sym}.
 \item Corollary 2.9, p. 241 \cite{mc} needs the additional assumption that $p|_R$ is the type corresponding to $C$ and $C$ has property $\ast$.
 \item The assumption that $(R,V)$ has property $\ast$ must be added to the hypotheses of Lemma 3.1 and appeals to Corollary 2.9 are replaced by appeals to this hypothesis.
 
 \item Note that the cut of type $p$ in $R\langle a \rangle$ also has property $\ast$ by Lemma \ref{superstructures} of this paper.
 
 \item In the proof of Lemma 3.1, p. 242, lines -18 to -15: The field assumption is used to reduce 4 cases to 2. But the two remaining cases can simply be proved in the same fashion as Case 2.
 \item In the proof of Lemma 3.1, p. 243, lines 1 and 2: ``Note that since $V$ is a group, ...'' should be replaced with ``Note that since $(C,R^{>C})$ is a valuational cut, ...'' and on page 243, lines 4,5 ``But $\frac{1}{2}\beta$ is also greater than every element of $V$, ...'' should be replaced with ``But $\beta - \epsilon$, where $\epsilon >0$ is such that $r-c>\epsilon$ for all $r\in R^{>C}$ and all $c\in C$, is also greater than every element of $C$, ...''. 
 
 \item The hypothesis of satisfying property $\ast$ must be added to the hypotheses of Lemma 3.2 and Theorem 3.3 and appeals to Lemma \ref{superstructures} (of this paper) added as appropriate to their proofs.

 \end{enumerate}
\end{proof}

For readers not already familiar with the proofs in \cite{mc}, this may be a significant amount of work, especially as the proofs in \cite{mc} themselves require one to go through the proofs in \cite{bp} to confirm results implicit, but not stated, therein.  We have therefore included a new, simplified proof of Theorem 3.3 of \cite{mc} in the remainder of this section which includes an easy proof of the results implicit in \cite{bp} that we require.  This new proof also contains a result that may be of independent interest, namely Proposition \ref{equivalent to valuational}, a characterization of the divide between valuational and non-valuational cuts in neostability theoretic terms.

In order to prove Theorem \ref{elextnew}, we will need a result on quantifier elimination for traces (Proposition \ref{poizat}) together with statement about the uniformity of this quantifier elimination (Remark \ref{poizat_uniformity}), both of which may be extracted from the proofs of Baisalov and Poizat  in \cite{bp}. 

We shall need the notion of separation which was introduced in \cite{bp}:

\begin{definition}\label{separated}
Let $a \in R^m$, $b \in R^n$, $A\subseteq R$, and let $p\in S_1(A)$.  Then $a$ and $b$ are {\em separated in $p$\/ over $A$} if either $$\dcl_{\mathcal{L}_0}(aA)\cap p(R) < \dcl_{\mathcal{L}_0}(bA)\cap p(R) \mbox{ or } \dcl_{\mathcal{L}_0}(bA)\cap p(R) < \dcl_{\mathcal{L}_0}(aA)\cap p(R).$$
We say that $a$ and $b$ are {\em $A$-separated\/} if they are separated in all one-types over $A$.  Note that if $a$ and $b$ are separated in $p$ and $\dcl_{\mathcal{L}_0}(aA)\cap p(R) \neq \emptyset$ then $a$ and $b$ are $A$-separated.
\end{definition} 

We are able to provide short proofs by limiting ourselves to the case where the externally definable set is a trace of a formula whose parameters form a Morley sequence in a cut satisfying $\ast$, using our result (Corollary \ref{sym}) on symmetry of dividing in Morley sequences, thus easily obtaining separated tuples.  (This is also Corollary 2.9 of \cite{mc}, although there it appeared (incorrectly) without the hypothesis $\ast$.)

\smallskip\noindent

\begin{corollary}\label{separation}
 Let $p \in S_1 (\mathcal{R})$ be an $R$-invariant extension of the $\mathcal{L}_0$-type over $R$ which states that $x$ realizes the cut $C$, and assume that latter type has property $\ast$.  If $t_1, \dots , t_n$ is a finite Morley sequence in $p$ over $R$ and $1 \leq k <n$, then $t_{1},\dots ,t_{k}$ is $R$-separated from $t_{k+1},\dots  ,t_{n}$.
\end{corollary}

\begin{proof}
This is simply Proposition \ref{sym} together with the definition of separation.
\end{proof}

\noindent In addition, we shall use the following fact:
\begin{fact}\label{cut}
 Let $a_1, \dots, a_k$ be such that each $a_i$ realizes a cut in $R\langle a_1, \dots a_{i-1}\rangle$.  Then every element of $R\langle a_1, \dots a_k\rangle \setminus R$ realizes a cut in $R$.  
\end{fact}
\begin{proof}
 This is clear when $k=1$.  Suppose the statement is false, and consider the minimal $k$ for which it fails.  Let $b = f(a_1, \dots, a_k)$ witness its failure.  Replacing $b$ with $1/b$, if necessary, we may assume that there is a closest element of $R$ to $b$.  Call this element $r$.  Since $b$ lies in a cut in $R\langle a_1, \dots a_{k-1}\rangle$, there must be elements of $R\langle a_1, \dots a_{k-1}\rangle$ between $b$ and $r$.  But these elements must therefore realize a definable type over $R$, contradicting the minimality of $k$.
\end{proof}

Finally, we will need to take a global type, $p$, invariant over a small model $R$ where both the invariant type and its restriction to $R$ are non-definable, build a Morley sequence $t_1, \dots, t_n$ in $p$ over $R$, and have that $p|_{Rt_1\dots t_n}$ is also nondefinable.  The following example shows that this is not always the case:

\begin{example}
 Let $A$ be the collection of rational numbers less than $\pi$ and let $p$ be the global type implied by $\{x>a \;|\;a\in A\} \cup \{x<r \;|\; r>A\}$.  Clearly $p$ is invariant over the empty set and is valuational (as is any non-definable global type invariant over a small set).  If, however, $R$ is archimedean, then $p|_R$ is either a non-valuational cut or a non-cut.  Let $R$ be an archimedean model. If $t_1, t_2, \dots $ is a Morley sequence in $p|_R$ then the only realization of $p|_R$ in $R\langle t_1 \rangle$ is $t_1$ itself.  If $p|_R$ is a non-cut then, clearly, $\tp(t_1/R)$ is definable, and if $p|_R$ is a non-valuational cut, then $\tp(t_2/Rt_1)$ is definable.
\end{example}

On the other hand, if $p|_A$ is a valuational type and $B \supseteq A$, then $p|_B$ is also valuational (witnessed by the same $\epsilon$).  In particular, $t_1, \dots, t_k$ is a Morley sequence in $p|_A$,   then $p|_{At_1\dots t_k}$ is also valuational, and we have the following characterization of valuational cuts:

\begin{proposition}\label{equivalent to valuational}
 If $p$ is a global 1-type invariant over a small set $A$, and $B\supseteq A$, then the following are equivalent
\begin{enumerate}
 \item $p|_B$ is a valuational type.
 \item  $p|_B$ is a cut and whenever $\dcl(Bc)$ contains $t\models p|_B$, then $\dcl(Bc)$ contains more than one realization of $p|_B$.
  \item  For all $k$, if $t_1, \dots, t_k$ is a Morley sequence in $p|_B$ then $p|_{Bt_1\dots t_k}$ is nondefinable.
  \item If $t_1, t_2$ is a Morley sequence in $p|_B$, then $\tp(t_1/B)$ and $\tp(t_2/B t_1)$ are nondefinable.
\end{enumerate}

\end{proposition}

\begin{proof}
Since $p$ is a 1-type invariant over $B$, we may assume, replacing $B$ with an appropriate subset of $\dcl(B)$, that either $p$ is implied by $\{x>b\;|\;b\in B\}\cup \{x<r\;|\;r>B\}$ or $p$ is implied by $\{x<b\;|\;b\in B\}\cup \{x>r\;|\;r<B\}$.  For convenience, we assume the former.  The proof in the latter case is identical.

    \begin{trivlist}
        \item[$1. \implies 2.$] Let $\epsilon$ witness that $p|_B$ is valuational, and let $t\models p|_B$.  Then $t-\epsilon\models p|_B$.
        \item[$2. \implies 1. $] Suppose $p|_B$ is a cut, and let $c$ be such that $t, \tilde t \in \dcl(Bc)$ both realize $p|_B$.  Note that $|t-\tilde t|$ cannot realize the right infinitessimal neighborhood of 0 in $\dcl(B)$, since $p|_B$ is a cut.  So let $\epsilon$ be a positive element of $\dcl(B)$ less than $|t-\tilde t|$.  This $\epsilon$ witnesses that $p|_B$ is valuational.
        \item[$1. \implies 3.$] Again, let $\epsilon$ witness that $p|_B$ is valuational.  Consider any $a\in\dcl(Bt_1\dots t_k)$ with $a\models p$.  Then $a-\epsilon \models p$.  Thus there is no smallest element of $p|_B \cap \dcl(Bt_1\dots t_k)$.  Since $p|_B$ is in particular a cut, there is no largest element of $B$.  Thus, $p|_{Bt_1 \dots t_k}$ is not definable.                                                                                        
        \item[$3. \implies 4.$] Clear.
        \item[$4. \implies 1.$] Let $t_1, t_2$ be a Morley sequence in $p|_B$.  Since $t_2$ and $t_1-t_2$ are interdefinable over $Bt_1$, and $\tp(t_2/B_{t_1})$ is nondefinable, $\tp(t_1-t_2/Bt_1)$ is also nondefinable.  In particular, $t_1-t_2$ does not realize the right infinitessimal neighborhood of $0$ in $\dcl(Bt_1)$, and one may choose positive $\tilde \epsilon\in \dcl(Bt_1)$ such that $t_2+\tilde\epsilon<t_1$.  Since $\tp(t_1/B)$ is nondefinable, $\tilde \epsilon$ does not realize the right infinitessmal neighborhood of $0$ in $\dcl(B)$, and one may choose positive $\epsilon\in \dcl(B)$ with $\epsilon<\tilde\epsilon$.  This $\epsilon$ witnesses that $p|_B$ is valuational.

    \end{trivlist}
\end{proof}

In the following proposition, we let $p$ be a global 1-type which is invariant over the empty set and is such that $p|_R$ has property $\ast$.  Consider an $|R|$-saturated elementary extension $\widetilde{R}$ of $R$.  Up to this point, $\mathcal{L}$ has been $\mathcal{L}_0$ expanded by a predicate for a cut.  For the remainder of this section we work in more generality. Let $S$ and $\tilde S$ be the realizations of $\theta(x_1, \dots, x_k, t_1, \dots, t_n)$ in $R$ and $\widetilde{R}$ respectively, where the parameters from $R$ are suppressed and the $t_1, \dots, t_n \in \widetilde{R}$ form a Morley sequence in $p$ over $R$.  Expand $\mathcal{L}_0$ to $\mathcal{L}$ by adding a predicate $S$ defining, in a slight abuse of notation, the set $S$ in $R$.

\begin{proposition}\label{poizat}
 Let $\varphi$ be an $\mathcal{L}$-formula with parameters from $R$.  Then there is an $\mathcal{L}_0$-formula, $\tau_{\varphi}$, such that  $\varphi(R)=\tau_{\varphi}(\widetilde{R})\cap R^k$, where the parameters of $\tau_{\varphi}$ form a Morley sequence in $p|_R$.  
\end{proposition}

\begin{proof}
Let $\underline{x}=(x_1 ,\dots ,x_k )$.  
 It suffices to show that if $\varphi(x_0,\underline{x})$ is the trace of the $\mathcal{L}_0$-formula $\tau_{\varphi}(x_0, \underline{x}, t_1, \dots, t_n)$, then $\exists x_0 \varphi(x_0,\underline{x})$ is the trace of 
$$\tau_{\exists x_0 \varphi} = \exists x_0 (\tau_{\varphi}(x_0, \underline{x}, t_1, \dots, t_n) \land \tau_{\varphi}(x_0, \underline{x}, t_{n+1}, \dots, t_{2n})),$$ 
where $t_1, \dots, t_{2n}\in \widetilde{R}$ form a Morley sequence in $p$ over $R$.  Suppose $r_1, \dots, r_k \in R$ are such that there is $r_0\in R$ with $\varphi(r_0,r_1 , \dots, r_k)$.  Then \[\widetilde{R}\models \tau_{\varphi}(r_0, r_1, \dots, r_k, t_1, \dots, t_n),\] and since $t_1, \dots, t_n \equiv_R t_{n+1}, \dots, t_{2n}$, one has \[ \widetilde{R} \models \tau_{\varphi}(r_0, r_1, \dots, r_k, t_{n+1}, \dots, t_{2n})\] as well.

Now suppose that there is no $r_0\in R$ such that $\phi (r_0 ,r_1 , \dots ,r_k )$ and assume for a contradiction that there is $\tilde r_0\in \widetilde{R}$ with  $$\widetilde{R} \models \tau_{\varphi}(\tilde r_0, r_1, \dots, r_k, t_1, \dots, t_n) \land \tau_{\varphi}(\tilde r_0, r_1,\dots, r_k, t_{n+1}, \dots, t_{2n}).$$ 
We consider $\tp( \tilde r_0/Rt_1, \dots ,t_n)$.  Either $\tilde r_0 \in \dcl(Rt_1, \dots, t_n)$ or there is an interval $(a, b)$ containing $\tilde r_0$ with $a, b \in \dcl(Rt_1, \dots t_n)$ (possibly equal to $-\infty, \infty$) such that any $s \in (a,b)$ is such that $\tau_{\varphi}(s, r_1, \dots, r_k, t_1, \dots, t_n)$.  Note that as we are supposing that there is no element of $R$ in that interval, the entire interval realizes a single type $q\in S_1 (R)$.  Choose $\tilde r_1\in (a,b)$ with $\tilde r_1\in \dcl(Rt_1, \dots t_n)$.  As $p|_R$ is valuational, $\tp(t_i/Rt_1\dots t_{i-1})$ is a cut, by Proposition \ref{equivalent to valuational}.  This allows us to apply Proposition \ref{cut} to $\tilde r_1$, showing that $q$ is not a definable type.  Thus neither $a$ nor $b$ can be in $R$ or equal to $-\infty, \infty$.

Likewise when considering $\tp( \tilde r_0/Rt_{n+1}, \dots t_{2n})$ we obtain either $\tilde r_0\in \dcl(Rt_{n+1}, \dots t_{2n})$ or the existence of an interval $(c, d)$ analogous to $(a,b)$ above.  

Now note that we have contradicted the separation of $t_1, \dots, t_n$ from $t_{n+1},\dots, t_{2n}$ (which follows from Corollary \ref{separation}):  For either there is an element not in $R$ common to their definable closures, or else there is an element of the definable closure of one contained in an interval defined over the other.
\end{proof}

\begin{remark}\label{poizat_uniformity}
It is easy to see, inspecting the proof of Proposition \ref{poizat}, that the construction of $\tau_\varphi$ is uniform in the sense that it is independent of $p$, $R$, and $\widetilde{R}$.  To be precise, let $R_1 \preceq \widetilde{R}_1$ (in the sense of $\mathcal{L}_0$) with $\widetilde{R}_1$ sufficiently saturated, and let $S_1$ and $\widetilde{S}_1$ denote the realizations of $\theta(x_1, \dots, x_k,s_1, \dots s_n)$ in $S_1$ and $\widetilde{S}_1$ respectively with $s_1, \dots, s_n\in \widetilde{S}_1$ a Morley sequence in  $p_1|_{R_1}$, where $p_1$ is a global 1-type invariant over $\emptyset$ such that $p_1|_{R_1}$ has property $\ast$.  Then given any $\varphi$, the proof above creates the same $\tau_{\varphi}$ for $R_1$ and $p_1$ as it does for $R$ and $p$.
\end{remark}

Note that the uniformity observed in Remark \ref{poizat_uniformity} translates easily into a fact about elementary substructures.  Namely, we have the following lemma.

\begin{corollary} \label{elementary extension}
 Consider $R$ and $R_1$ as in Proposition \ref{poizat} where in addition $R\subseteq R_1$, $\widetilde{R} = \widetilde{R}_1$. Consider  $p=p_1$ invariant over $R$ such that $p|_R$ and $p|_{R_1}$ have property  $\ast$, and moreover $s_1,\dots, s_n = t_1, \dots, t_n$. Then $R \preceq R_1$ as $\mathcal{L}$-structures.  
\end{corollary}

\begin{proof}
 We may add constants for $R$ to $\mathcal{L}$ to make $p$ invariant over $\emptyset$ so that we may apply Proposition \ref{poizat}. Suppose that $R_1 \models \exists x \varphi(x, \vec r)$ where $\vec r \in R^m$.  This happens if and only if 
 $\vec r\in \tau_{\exists x \varphi}(\widetilde{R}) \cap R_{1}^{m}$, where $\tau_{\exists x \varphi}$ is the $\widetilde{R}$-definable set whose trace is $\exists x \varphi(x, \vec R_{1}^{m})$.  But the set $\exists x \varphi(x, R^m)$ is also the trace of $\tau_{\exists x \varphi}$.  Obviously $\vec r\in \tau_{\exists x \varphi}(\widetilde{R}) \cap R^m$.  Thus $R \models \exists x \varphi(x, \vec r)$, and we are done by the Tarski-Vaught test.
\end{proof}

We may now use this corollary to prove a slightly strengthened version (in that we only need to assume the cut is valuational, rather than an actual valuation ring) of Theorem 3.3 of \cite{mc}.  

First note that if $R\langle a \rangle$ is a superstructure of $R$ (as an $\mathcal{L}$-structure) and there is an element (which we may assume to be $a$) in $C^{>R}$ with $f(a)$ in the same cut in $R$ as $a$ but with $f(a)$ greater than $C$, then we cannot hope to satisfy the hypotheses of Corollary \ref{elementary extension}.  In particular, the requirement that $s_1 \dots, s_n = t_1, \dots, t_n$ can not be satisfied.  After all, $C$ is the trace of $x<t$ and $t$ must lie between $a$ and $f(a)$, but an $R$-invariant global type whose restriction to $R$ is the cut of $a$ in $R$ would produce a Morley sequence either less than $a$ or greater than $f(a)$.  Thus this Morley sequence could not include $t$.

\smallskip

\noindent However, this is the only obstruction:

\begin{lemma}\label{lemma3.2}[Lemma 3.2 \cite{mc}]  There is an $R$-invariant global type $q$ with $q|_R$ being the cut in $R$ given by $C$ and $q|_{\ra}$ being the cut in $\ra$ given by $C$, and there is a finite Morley sequence $t_1, \dots, t_n$  in $q$ over $\ra$ (and hence over $R$) in each of the following cases:
\begin{itemize}
 \item[$a)$]  $a$ realizes the cut in $R$ given by $C$ and $C(\ra)$ is the convex hull of $C(R)$ in $\ra$. 
\item[$b)$]   $a$ realizes the cut given by $C$ and $C(\ra)=\{  x\in \ra\;|\; x<R^{>C(R)}  \}.$  
\item[$c)$]   $a$ is such that $\dcl_R (a)$ does not realize the cut given by $C$ (hence $C(\ra)$ is the convex hull of $C$ in $\ra$ and $C(\ra)=\{  x\in \ra\;|\; x<R^{>C(R)}  \}$).
\end{itemize}
\label{elext}
\end{lemma}
\begin{proof} In a), we may take $q$ to be the global type implied by $$\{x>r\;|\;r\in C(R) \} \cup \{x<r\;|\; r\in \mathcal{R} \textrm{ and } r>C(R)\}.$$  In b), we let $q$ be the global type implied by $$\{x<r \;|\; r\in R^{>C(R)}\} \cup \{x>r \;|\; r \in \R \textrm{ and }r< R^{>C(R)} \}.$$  In c) we may do either.

\end{proof}
\noindent Now we obtain our criterion for elementary extensions:

{
\renewcommand{\thetheorem}{\ref{elextnew}}
\begin{theorem}\label{eext}[Theorem 3.3 \cite{mc}] Let $a\in \mathcal{R}$ and assume that $(R,C)$ has property $\ast$. Then $(\ra ,C)$ fails to be an elementary extension of $(R,C)$ if and only if there are $R$-definable one-variable functions $f$ and $g$ with $C(R)<f(a),g(a)<R^{>C(R)}$ and $f(a)\in C$ and $g(a)>C$.
\end{theorem}
\addtocounter{theorem}{-1}
}

\begin{proof}
First assume there are no functions $f, g$ as in the statement of the theorem.  Then we may apply Lemma \ref{superstructures} to obtain that $(\ra ,C)$ has property $\ast$.  The lack of such $f, g$ also allows us to apply Lemma \ref{lemma3.2} to find an $R$-invariant $q$ with $q|_R$ being the cut in $R$ given by $C$ and $q|_{\ra}$ being the cut in $\ra$ given by $C$.  Now we have satisfied the hypotheses of Corollary \ref{elementary extension}, and conclude that $(R,C)\preceq (\ra, C)$.

Now assume that there are such functions $f$ and $g$, and define $h$ to be $g \circ f^{-1}$.  Note that since $h$ maps $f(a)$ to $g(a)$, $h$ maps the cut in $R$ corresponding to $C$ to itself.  Thus $h$ is increasing (see Fact \ref{decreasing}) on an $R$-definable interval containing this cut.  Thus  \[(\ra, C)\models \textrm{for all sufficiently large } x \in C,\; h(x)>C,\] hence \[(R,C) \models \textrm{for all sufficiently large } x\in C, \; h(x)>C. \] Take $r\in C(R)$ with $h(r)\in R^{>C(R)}$.    But $r<f(a)$ while $h(r)>h(f(a))$, a contradiction to $h$ being increasing.
\end{proof}
 
\end{section}

\begin{section}{Substructures of models of $T$ and quantifier elimination}\label{substrqesection}

In this section, we obtain our quantifier elimination, universal axiomatization, and definable Skolem function results.  The key result is that for a suitable $(R_0,C)\preceq (R,C)$, any single element $a\in R$ results in $(R_0,C)\preceq (R_0\langle a\rangle, C)\preceq(R,C)$.  We will obtain rather easily that $(R_0,C)\preceq (R_0\langle a\rangle, C)$, but to continue to inductively build up to $(R,C)$ we will also need that $(R_0\langle a\rangle, C)\preceq(R,C)$, and this will take more work, in particular in the case where $(R_0\langle a\rangle, C)$ contains no elements realizing the cut in $R_0$ corresponding to $C$.

\begin{lemma}\label{basestep} 
Suppose  $(R_0 , C )\preceq (R,C)$, $C$ defines a cut in $R_0$ with property $\ast$, and $a\in R\setminus R_0$. 
Then $(R_0 , C ) \preceq (R_0\langle a \rangle , C )$.  
\end{lemma}

\begin{proof}  
If not, then, by Theorem \ref{elextnew}, we may assume that $a$ realizes the type
\[C(R_0)<x<R_{0}^{>C(R_0)},\]
and there is an $\mathcal{L}_0$-definable over $R_0$ function $f$ such that $f(a)$ also realizes the same type, and moreover $a\in C(R_0\langle a \rangle)$ and $f(a)>C(R_0\langle a \rangle)$.  Hence \[(R,C)\models \exists x \in C \; f(x)>C.\]  By Fact \ref{decreasing}, $f$ is increasing, and so $f(C(R_0 ))\subseteq C(R_0 )$ and $f(R_{0}^{>C(R_0 ) }) \subseteq R_{0}^{>C(R_0 )
}$, contradicting $(R_0 , C )\preceq (R,C)$. 
\end{proof}

\begin{lemma}\label{mainlemma}
Suppose $(R_0 ,C)\preceq (R,C)$, and suppose $C$ defines a cut in $R_0$ with property $\ast$. Let $a_1 , \dots ,a_n \in R$ and set $R_n =R_0\langle a_1, \dots, a_n \rangle$ and $C_n = C(\mathcal{R}) \cap R_n$.  
Then $(R_0 , C) \preceq (R_n, C)$.  Moreover, suppose $a_1, \dots, a_n $ is an independent tuple such that, for some $m\leq n$, $a_{i+1} \in C_{i+1}^{>C_i} $ for each $i\leq m$, and $C_m$ is cofinal in $C_n$.  Then for all $i$, $(R_i, C) \preceq (R_{i+1},C)$.
\end{lemma}

\begin{proof}
We may as well assume that $a_1 , \dots ,a_n$ are as in the ``Moreover, ..." part of the statement of the lemma.  
Note that once we have shown $(R_i , C)\preceq (R_{i+1},C)$, then, by Lemma \ref{elextnew} and Lemma \ref{superstructures}, we have that $C$ defines a cut in  $R_{i+1}$ which has property $\ast$, and hence we are at liberty to use Lemma \ref{elextnew} at this stage, too.

 We start by showing that $(R_i, C)\preceq(R_{i+1},C)$ for $i<m$, which we do by induction.

\smallskip

\noindent $\mathbf{i=0}$: This is Lemma \ref{basestep}.
\smallskip

\noindent \textbf{Inductive Step}: Towards a contradiction, assume there is an $R_0$-definable function $f$ such that $f(a_1 , \dots ,a_i, x_{i+1})$ maps the set of realizations of the type

\[C_i <x<R_{i}^{>C_i}\]
to itself and $f(a_1, \dots, a_i, a_{i+1})$ is greater than $C_{i+1}$.   Since each $a_{j+1}>C_j$ for $j<i$, we have, by Lemma \ref{elextnew}, that $R_{0}^{>C_0}$ is coinitial in $R_{i}^{>C_i}$.

Note that we cannot repeat the argument of Lemma \ref{basestep} exactly, for while $(R_0, C)\preceq (R_i, C)$, we do not know $(R_i, C)\preceq (R, C)$.  But as $(R_0, C)\preceq (R, C)$, we have that for any $r_{01} , \dots ,r_{0i} \in C_0$, 
$$ \begin{array}{ll}
(R,C) & \models  \exists x_{1} \in C^{>r_{01}} \dots  \exists x_{i} \in C^{>r_{0i}}\; x_1<\dots<x_i \\
& \exists y \in C^{>x_i} \, \forall x_{i+1}\in C^{>y} f(x_1, \dots, x_i, x_{i+1}) >C \\
& \textrm{and } f \textrm{ is increasing as a function of } x_{i+1} \textrm{ on }C^{>y}.
\end{array}
$$
Since $(R_0 , C ) \preceq (R,C)$, $(R_0, C)$ satisfies the same $\mathcal{L}_{R_0}$-formulas.  Thus it also satisfies
$$ \begin{array}{ll}
(R_0,C) & \models \forall z_1 \in C \dots \forall z_i \in C \, \exists x_{1} \in C^{>z_1 } \dots  \exists x_{i} \in C^{>z_i } x_1<\dots<x_i \\
& \exists y\in C^{>x_i} \, \forall x_{i+1}\in C^{>y} f(x_1, \dots, x_i, x_{i+1}) >C\\
& \textrm{and } f \textrm{ is increasing as a function of } x_{i+1} \textrm{ on }C^{>y}.
\end{array}
$$

\noindent Now since $(R_0 , C)\preceq (R_1, C)$, $(R_1 , C)$ satisfies the same sentence. So we may choose $a'_1\in C_{1}^{\geq a_1}$ (and thus $\tp_{\mathcal{L}_{0}}(a'_1/R_0)=\tp_{\mathcal{L}_{0}}(a_1/R_0)$) such that 

$$ \begin{array}{ll}
(R_1,C) & \models \forall z_2 \in C \dots \forall z_i \in C \,
\exists x_{2} \in C^{>z_2} \dots  \exists x_{i} \in C^{>z_i} a'_1<x_2<\dots<x_i \\
& \exists y\in C^{>x_i} \, \forall x_{i+1}\in C^{>y} f(a'_1, x_2, \dots, x_i, x_{i+1}) >C \\
& \textrm{and } f \textrm{ is increasing as a function of } x_{i+1} \textrm{ on } C^{>y}.
\end{array}
$$

\noindent Similarly, we may find $a'_2 \in C_{2}^{\geq a_2}$ (and thus $\tp_{\mathcal{L}_{0}}(a'_1 a'_2 /R_0)=\tp_{\mathcal{L}_{0}}(a_1 a_2 /R_0)$) such that 

$$ \begin{array}{ll}
(R_2,C) & \models \forall z_3 \in C \dots \forall z_i \in C \, \exists x_{3} \in C^{>z_3} \dots  \exists x_{i} \in C^{>z_i} a'_1<a'_2 <\dots<x_i \\
& \exists y\in C^{>x_i } \, \forall x_{i+1} \in C^{>y} f(a'_1, a'_2, x_3, \dots, x_i, x_{i+1}) >C \\
& \textrm{and } f \textrm{ is increasing as a function of } x_{i+1} \textrm{ on } C^{>y}.
\end{array}
$$

\noindent Continuing in this fashion, we find that 

$$ \begin{array}{ll}
(R_i,C) & \models \exists y \in C^{>a'_i} \, \forall x_{i+1}\in C^{>y} \, f(a'_1, \dots, a'_i, x_{i+1}) >C \\
& \textrm{and } f \textrm{ is increasing as a function of } x_{i+1} \textrm{ on } C^{>y}.
\end{array}
$$
\noindent Note that $\tp_{\mathcal{L}_0} (\underline{a}'/R_0)=\tp_{\mathcal{L}_0}(\underline{a}/R_0)$, where $\underline{a}'=(a'_1 , \dots , a'_i)$ and $\underline{a}=(a_1 ,\dots , a_i)$.  Furthermore, the above sentence is witnessed by $y=h(\underline{a}')$ for some $R_0$-definable function $h$.

Let $g$ be another $R_0$-definable function such that $h(\underline{a}')<g(\underline{a}')$ and $g(\underline{a}')\in C_i$ (so $f$ is increasing at $g(\underline{a'})$).  Since $f(\underline{a}', g(\underline{a}')) >C_i$, $f(\underline{a}', g(\underline{a}')) >r_0$ for some $r_0 \in R_{0}^{>C_0}$ by coinitiality of $R_{0}^{>C_0 }$ in $R_{i}^{>C_i}$. Then $g(\underline{a})<a_{i+1}$ (again by coinitiality of $R_{0}^{>C_0 }$ in $R_{i}^{>C_i}$) but $f(\underline{a},a_{i+1})<r_0 <f(\underline{a},g(\underline{a}))$, a contradiction with $f$ being increasing in $x_{i+1}$ on an interval containing $g(\underline{a})$ and $a_{i+1}$. This finishes the inductive argument.

Finally $(R_m, C)\preceq (R_n, C)$, by Theorem \ref{elextnew}.
\end{proof}

\begin{remark}\label{reverse order}
Note that instead of choosing $a_1, \dots, a_m\in C_n$ successively larger realizations of the type $C_0 <x<R_{0}^{>C_0}$
until $C_m$ was cofinal in $C_n$, we could just as easily have chosen successively smaller realizations $a_1, \dots, a_k$ of $C_0 <x<R_{0}^{>C_0}$, until $R_{k}^{>C_{k}}$ was coinitial in $R_{n}^{>C_n}$.  Then the inductive argument above is easily modified to yield $(R_{i},C)\preceq (R_{i+1}, C)$ for $i<k$, and Lemma \ref{elextnew} shows $(R_k, C)\preceq (R_n, C)$.
\end{remark}
We will use the above lemma to show that if one builds up from $(R_0,C)$ to $(R,C)$ by taking successively larger elements of $C(
R)$, each realizing the cut determined by $C$ in the previous model, until one has a cofinal subset of $C(R)$ and then builds the rest of the way to $R$ in any fashion, then each step from $R_0$ to $R$ is an elementary extension.  More precisely:

\begin{lemma} \label{ordered}
Suppose  $(R_0 , C )\preceq (R,C)$ and suppose $C$ defines a cut in $R_0$ with property $\ast$. Let $(a_{\alpha})_{\alpha<\kappa} $ be such that  $R=R_0\langle(a_{\alpha})_{\alpha<\kappa}\rangle$, and set $R_{\alpha}=R_0\langle(a_{\beta})_{\beta<\alpha} \rangle$ and $C_{\alpha}=R_{\alpha}\cap C(\mathcal{R})$.  Assume further that 
\begin{enumerate}
    \item there is $\gamma \leq \kappa$ such that $a_{\alpha}\in C_{\alpha+1}^{>C_{\alpha}}$ for all $\alpha<\gamma$, and $C_\gamma$ is cofinal in $C(R)$.
    \item there is $\delta \leq \kappa$ such that $R_{\delta}^{>C_\delta}$ is coinitial in $R^{>C(R)}$ and for all $\alpha$ with $\gamma \leq \alpha < \delta$, $a_{\alpha}\in R_{\alpha+1}^{>C_{\alpha+1}}$ but $a_\alpha < R_{\alpha}^{>C_{\alpha}}$.
\end{enumerate} Then \[(R_\alpha,C) \preceq (R_{\alpha+1},C)\preceq (R,C)\] for  each $\alpha<\kappa$.
\end{lemma}

\begin{proof}  Let $\gamma$ be as indicated in the lemma. Note that $(a_\alpha)_{\alpha<\gamma}$ is a Morley sequence in  $q$ (in the sense of Lemma \ref{superstructures} with $R_0$ playing the role of $R$ ).

By Lemma \ref{basestep}, $(R_0 , C)\preceq (R_1 , C)$. Then by Lemma \ref{superstructures}, $C$ defines a cut in $R_1$ with property $\ast$. Suppose $\alpha_0<\gamma$ is the first ordinal such that $(R_{\alpha_0}, C)\npreceq (R_{\alpha_0+1}, C)$.

By applying Lemma \ref{basestep} and then Theorem \ref{elextnew} at successor ordinals less than $\alpha_0$, we see that $(a_\alpha)_{\alpha<\alpha_0}$ satisfies condition (2) of Lemma \ref{superstructures}.  Thus in $(R_{\alpha_0}, C)$, one has that $C_{\alpha_0}$ satisfies property $\ast$ and we may apply Theorem \ref{elextnew} to obtain $f(a_{\alpha_0})$ not in $C_{\alpha_0 +1}$ but smaller than $R_{\alpha_0}^{>C_{\alpha_0}}$.  Suppose the parameters in $f$ are $a_{\beta_0}, \dots ,a_{\beta_n}$, where each $\beta_i < \alpha_0$.  Then \[R_0\langle a_{\beta_0}, \dots a_{\beta_n} \rangle\npreceq R_0\langle a_{\beta_0}, \dots a_{\beta_n}, a_{\alpha_0}\rangle ,\] contradicting Lemma \ref{mainlemma}.  

Note $(a_\alpha)_{\gamma\leq \alpha<\delta}$ is a Morley sequence in  $p$ (in the sense of Lemma \ref{superstructures}) over $R_\gamma$, showing that the cut $C_\alpha$ in $R_\alpha$ has property $\ast$ for each $\gamma<\alpha<\delta$, by condition (1) of Lemma \ref{superstructures}. Also note that  $(a_\alpha)_{\delta\leq \alpha}$ satisfies condition (3) of Lemma \ref{superstructures} (where $R_\delta$ plays the role of $R$ in Lemma \ref{superstructures}), showing that property $\ast$ holds for the cut $C_\alpha$ in $R_\alpha$ for each $\delta<\alpha$.

For any $\alpha\geq \gamma$, the fact that $C_\gamma$ is cofinal in $C(R)$ means that Theorem \ref{elextnew} implies $(R_\alpha, C)\preceq (R_{\alpha+1},C)$.  Taking unions at limit ordinals yields $(R_\beta, C)\preceq (R, C)$ for any $\beta<\kappa$.
\end{proof}

\begin{remark}\label{reverse order 2}
Suppose instead of having the sequence $(a_\alpha)_{\alpha<\kappa}$ in the order described in the hypotheses of Lemma \ref{ordered}, one has \begin{enumerate}
    \item $\gamma \leq \kappa$ such that $R_{\gamma}^{>C_\gamma}$ is coinitial in $R^{>C(R)}$ and for all $\alpha$ with $\alpha < \gamma$, $a_{\alpha}\in R_{\alpha+1}^{>C_{\alpha+1}}$ but $a_\alpha < R_{\alpha}^{>C_{\alpha}}$, and 
    \item $\delta \leq \kappa$ such that $a_{\alpha}\in C_{\alpha+1}^{>C_{\alpha}}$ for all $\gamma\leq\alpha<\delta$, and $C_\gamma$ is cofinal in $C(R)$.
\end{enumerate}
then one can see that the conclusion of Lemma \ref{ordered} still holds using the same proof except flipping the use of condition (1) and condition (2) of Lemma \ref{superstructures} and obtaining a contradiction with Remark \ref{reverse order} instead of Lemma \ref{mainlemma}.
\end{remark}

\begin{lemma}\label{orthogonal to C}
Suppose $(R_1, C)\preceq (R_2, C)$ and suppose that $(R_1, C)$ has property $\ast$.  Let $C_1=C(R_1)$.  Suppose that $a\in R_2$ with $C_1<a<R^{>C_1}$ and $C(a)$.  Suppose that $b\in R_2\setminus R_1$ and there is no $\tilde b \in\dcl_{R_1}(b)$ with $C_1<\tilde b < R^{>C_1}$.  Then $(R_1\langle b \rangle, C)\preceq (R_1\langle ab \rangle ,C)$.
\end{lemma}

\begin{proof}
First note that by Lemma \ref{basestep}, $(R_1, C)\preceq (R_1\langle a \rangle, C)$ and $(R_1, C)\preceq (R_1\langle b \rangle, C)$, and applying Lemma \ref{superstructures}, we see that $(R_1\langle a \rangle, C)$ and $(R_1\langle b \rangle, C)$ have property $\ast$.  Also note that $a$ can be taken to be the beginning of a sequence as in the hypotheses of Lemma \ref{ordered}, and so, applying that lemma, we see that $(R_1\langle a \rangle, C)\preceq (R_2, C)$.  Thus applying Lemma \ref{basestep}, we see that $(R_1\langle a \rangle, C)\preceq (R_1\langle ab \rangle, C)$.

Applying Theorem \ref{elextnew} to $(R_1\langle a \rangle, C)\preceq(R_1\langle ab \rangle, C)$, we see that, \textit{a priori}, there are three possibilities:
\begin{enumerate}
    \item $\dcl_{R_1\langle a \rangle}(b)$ contains no elements greater than $C(R_1\langle a \rangle)$ and less than $R_1\langle a \rangle^{>C(R_1\langle a \rangle)}$,
    \item The elements of $\dcl_{R_1\langle a \rangle}(b)$ that are greater than $C(R_1\langle a \rangle)$ and less than $R_1\langle a \rangle^{>C(R_1\langle a \rangle)}$ are all in $C$,
    \item The elements of $\dcl_{R_1\langle a \rangle}(b)$ that are greater than $C(R_1\langle a \rangle)$ and less than $R_1\langle a \rangle^{>C(R_1\langle a \rangle)}$ are all greater than $C$.
\end{enumerate}
In the first two cases, Theorem \ref{elextnew} implies that $(R_1\langle b \rangle, C)\preceq (R_1\langle ab \rangle)$, while in the third case it implies that $(R_1\langle b \rangle, C)\npreceq (R_1\langle ab \rangle)$.

Thus, for a contradiction, we assume that there is a $R_1$-definable function $f(x,y)$ with $C(R_1\langle a \rangle)<f(a,b)<R_1\langle a \rangle^{>C(R_1\langle a \rangle)}$.  Let $B_x=\{y|f(x,y)>C\}$.   Suppose first that $f(a,y)$ is increasing as a function of $y$.  Then $B_{a}$, realized as a subset of $R_1\langle a \rangle$, includes everything to the right of the cut of $b$ in $R_1\langle a \rangle$, since $f(a,b)>C(R_1\langle a \rangle)$.  (And note that $\tp(b/R_1\langle a \rangle)$ is a cut since $f(a, y)$ maps this type to the cut corresponding to $C$ in $R_1\langle a \rangle$.)  $B_a$ cannot include any $d\in R_1$ to the left of the cut of $b$, since then \[C(R_1\langle a \rangle) < f(a,d) < f(a,b) < R_1\langle a \rangle^{>C(R_1\langle a \rangle)},\] which is impossible as $f(a,d)\in R_1\langle a \rangle$.  And since $ R_1\langle b \rangle $ contains nothing realizing $\tp(a/R_1)$, $R_1\langle a \rangle$ can contain no points realizing $\tp(b/R_1)$.  Thus if $B_a$ included anything to the left of the cut of $b$ in $R_1\langle a \rangle$, it would contain a element of $R_1$. Similarly if $f(a,y)$ is decreasing, $B_{a}$ contains everything to the left of the cut of $b$ and contains nothing to its right.
 
Note that $f(x,b)$ maps $\tp(a/R_1)$ to itself.  Thus $f(x,b)$ is increasing on an definable neighborhood of $a$.  So there is an interval $I$ defined over $R_1$ and including $b$ so that for all $y\in I$, $f(x,y)$ is increasing in $x$. We restrict ourselves to such an $I$.
  
Since $f(x,y)$ is increasing as a function of $x$, $x_1<x_2$ implies that $B_{x_1}\subseteq B_{x_2}$.  Choose $\tilde a\in C(R_1\langle a \rangle)$ with $\tilde a>a$.  $B_{a}\subseteq B_{\tilde a}$ so $B_{\tilde a}$ also includes everything on one side of the cut of $b$, and we claim that $B_{\tilde a}$ also cannot contain any $d$ on the other side of the cut of $b$.  If it did contain a $d$ on the other side of $b$, then we could choose such a $d$ in $R_1$.  Then $(R_1\langle a \rangle ,C)$ would satisfy the sentence ``there is x in $C$ such that $B_x$ contains $d$", and  $(R_1 ,C)$, being an elementary submodel, would satisfy the same sentence. Thus $R_1$ would contain an $r$ in $C$ such that $d\in B_{r}$.  Thus $f(r,d)>C$, but $f(r,d)<f(a,d)<f(a,b)$, so $f(r,d)\in R_1$ would realize the cut of $a$ in $R_1$.

So for any $\tilde a>a$ in $C(R_1\langle a \rangle)$, $B_{\tilde a}=B_{a}$.  Thus \[(R_1\langle a \rangle, C) \models \exists x_0 \bigcup_{x\in C} B_x = B_{x_0}.\]  Since $(R_1, C)\preceq(R_1\langle a \rangle, C)$, we see that $(R_1, C)\models \exists x_0 \bigcup_{x\in C} B_x = B_{x_0}$. 
  
Say $a_0\in R_1$ is such an $x_0$.  $(R_1, C)$ satisfies $\bigcup_{x\in C} B_x = B_{a_0}$, so $(R_1\langle a \rangle, C)$ does as well.  Thus $(R_1\langle a \rangle, C) \models B_{a} = B_{a_0}$.  Since $(R_1\langle a \rangle, C)\preceq (R_1\langle ab \rangle, C)$, $B_{a} = B_{a_0}$ holds in $(R_1\langle ab \rangle, C)$ as well.  In particular, $f(a_0,b)>C$. Since $f(a_0,b) >C$ is quantifier free, it holds in $(R_1\langle b \rangle, C)$ as well.  As we also have $f(a_0,b)< f(a,b)$, we see that $C_1<f(a_0,b) < R^{>C_1}$, contradicting our choice of $b$.

\end{proof}

\begin{lemma}\label{elsubstr}
Suppose $(R_0 , C)\preceq (R,C)$, $C$ defines a cut in $R_0$ with property $\ast$, and $b\in R\setminus R_0$.  Then $(R_0\langle b\rangle , C)\preceq (R,C)$.
\end{lemma}

\begin{proof}
\begin{trivlist}
\item[Case 1:] There is an element of $\dcl_{R_0}(b)$ in $C(R)^{>C(R_0)}$:  
We may assume that this element is $b$.  We
may find $(a_{\alpha})_{\alpha <\kappa}$ as in the hypotheses of Lemma \ref{ordered} with $a_0 = b$ to obtain $(R_0\langle b \rangle ,C) \preceq (R, C)$. 
\item[Case 2:] There is an element of $\dcl_{R_0}(b)$ greater than $C(R_0\langle b \rangle)$ but less than $R_0^{>C(R_0)}$:
As in Case 1, but using Remark \ref{reverse order 2}.
 
\item[Case 3:] No element of $\dcl_{R_0}(b)$ is greater than $C(R_0)$ but less than $R_0^{>C(R_0)}$. Choose $(a_{\alpha})_{\alpha <\kappa}$ as in the hypotheses of Lemma \ref{ordered} (and let $R_{\alpha}$ and $\gamma$ be as in the proof of Lemma \ref{ordered} as well) with the additional requirement that if at any stage $\alpha$ there is an element of $\dcl_{R_0}(b(a_{\beta})_{\beta <\alpha })$ that could be chosen as $a_{\alpha}$ then one chooses such an element.   If this happens at some stage prior to $\gamma$, let $\beta$ denote that stage.  Otherwise, let $\beta=\gamma$. Thus one has:

\begin{tikzcd}
(R_0 \langle b\rangle, C) \arrow[r, hook, "\rho_1"] 
& (R_1\langle b \rangle, C)  \arrow[r, hook, "\rho_2"] &\dots  \arrow[r, hook, "\rho_{\beta}"]& (R_{\beta}\langle b\rangle, C) \arrow[r, hook, "\rho"] & (R,C)\\
(R_0,C) \arrow[r, hook, "\tau_1"] \arrow[u, hook, "\sigma_0"]
& (R_1,C) \arrow[r, hook, "\tau_2"] \arrow[u, hook, "\sigma_1"] & \dots \arrow[r, hook, "\tau_{\beta}"] & (R_{\beta},C) \arrow[u, hook, "\sigma_{\beta}"] & 
\end{tikzcd}

\noindent where each $\tau_i$ is an elementary embedding by Lemma \ref{ordered}, each $\sigma_i$ is an elementary embedding by Lemma \ref{basestep}, and $\rho$ is an elementary embedding since either $\beta<\gamma$ and $\rho$ continues the chain of elementary embeddings $(R_{\alpha}, C)\preceq(R_{\alpha+1}, C)$, or $\beta=\gamma$ and $C(R_{\beta}\langle b\rangle)$ is cofinal in $C(R)$.

Note that $R_0\langle b \rangle$ contains no $x$ with $C(R_0)<x<R_0^{C(R_0)}$,  so Lemma \ref{orthogonal to C} shows that $\rho_{1}$ is an elementary embedding.  Inductively, except for possibly $i=\beta$, $R_i\langle b \rangle$ contains no $x$ with $C(R_i)<x<R_i^{C(R_i)}$,  so Lemma \ref{orthogonal to C} shows that $\rho_{i+1}$ is an elementary embedding.  Thus composing the $\rho_i$ and $\rho$, one sees that $(R_0\langle b\rangle, C)\preceq (R,C)$.

\end{trivlist}

\end{proof}

\begin{remark}
Up to now, we have worked inside a fixed monster model of $Th(R,C)$, and so we have treated property $\ast$ as a property of a particular substructure even though its definition depends on parameters from outside the structure.  In what follows, we will take a substructure of $(R,C)$ and embed it into a model of a potentially different theory, raising the question of whether this could potentially change whether the substructure satisfies $\ast$.  

But it is easy to see that this will not happen.  For if a substructure $(R_0,C)\subseteq (R,C)$ does not satisfy $\ast$ as a substructure of $(\mathcal{R},C)$, then there is a tuple $t$ and a function $f(x,t)$ mapping a cofinal sequence in $C(R_0)$ to a sequence coinitial in $R_0^{>C(R_0)}$, and this can be seen to be a property of the $\mathcal{L}_0$-type of $t$ over $R_0$.  If $(R_0,C)$ is also a substructure of a sufficiently saturated $(\mathcal{R}_1, C)$, then as the $\mathcal{L}_0$-theories of $\mathcal{R}_1$ and $\mathcal{R}$ are the same, this type is also realized in $\mathcal{R}_1$, and $(R_0,C)$ also fails to satisfy $\ast$ as a substructure of $\mathcal{R}_1$.
\end{remark}

\begin{proposition}\label{substr}
Suppose $(R_0 ,C)$ is a substructure of $(R,C)$ such that $(R,C)$ can be embedded over $R_0$ into an elementary extension $(\widetilde{R}, C)$ of $(R_0, C)$, and suppose $C$ defines a cut in $R_0$ with property $\ast$.  Then $(R_0, C)\preceq (R,C) \preceq (\widetilde{R}, C)$.  
\end{proposition}

\begin{proof}
This follows from Lemmas \ref{basestep} and \ref{elsubstr} by transfinite induction.
 \end{proof}

The above shows that, after adding constants for an elementary submodel of $(R,C)$ where the cut $C$ has property $\ast$, any substructure of $(R,C)$ is an elementary submodel.  This easily leads to quantifier elimination in the language with the new constants. 

We now give a more concrete description of the constants needed.  We use the following fact, found for example in Hodges \cite{h}, p. 294, Theorem 6.5.1 (where his $B$ and $\bar a$ are both $A$ below).

\begin{fact}\label{Hodges}
 Let $A$ be a substructure of $C$, and assume that every existential formula with parameters from $A$ that is true in $C$ is also true in $A$.  Then there is an embedding of $C$ into an elementary extension, $D$, of $A$, and this embedding can be chosen to be the identity on $A$.
\end{fact}

\begin{lemma} \label{existential theory}
Fix a model $(R,C)$ where the cut $C$ has property $\ast$.  Choose a sequence $(c_i)_{i\in I}$ cofinal in $C$ and a sequence $(d_j)_{j\in J}$ coinitial in $R^{>C}$.  Let $R_0$ be $\dcl_\mathcal{L}((c_i)_{i\in I} (d_j)_{j\in J})$.  Then  $(R_0, C)$ has property $\ast$, and $(R,C)$ can be embedded over $R_0$ into an elementary extension $(\widetilde{R}, C)$ of $(R_0, C)$.
\end{lemma}
\begin{proof}
Note that by Proposition \ref{equivalent to star} (1), $(R_0,C)$ clearly satisfies $\ast$.  By Fact \ref{Hodges}, it suffices to show that any existential statement with parameters from $R_0$ true in $(R,C)$ is also true in $(R_0, C)$.  Note that any quantifier free formula in $\mathcal{L}_C$ is a finite disjunction of formulas of the form \[\theta(x)\land C(f_1(x))\land \dots \land C(f_n(x)) \land \dots \land \neg C(g_1(x))\land \dots \land \neg C(g_m(x))\] where $\theta$, each $f_i$, and each $g_j$ are $\mathcal{L}_0$-definable over $R_0$, and where $x$ may be a tuple. By setting \[f(x)=max(f_1(x),\dots, f_n(x))\mbox{ and } g(x)=min(g_1(x),\dots, g_m(x)),\] we see that every quantifier free formula is a finite disjunction of formulas \[\varphi(x) = \theta(x)\land C(f(x))\land \neg C(g(x)).\]
Note that in $(R,C)$, $C(x)$ iff $\bigvee_{i\in I}x<c_i$, and $\neg C(x)$ iff $\bigvee_{j\in J}x>d_j$. Thus, \[\varphi(R)= \theta(R)\cap \bigcup_{i\in I}f^{-1}(R^{<c_i})\cap \bigcup_{j\in J}g^{-1}(R^{>d_j}).\]  
Now assume that $(R,C)\models \exists x \varphi(x)$ is witnessed by $r$.  In particular, there are $i, j$ such that $r$ satisfies the formula $\theta(x)\land f(x)<c_i \land g(x)>d_j$.  As this set is $\mathcal{L}_0$-definable over $R_0$, and $R_0$ considered as an $\mathcal{L}_0$-structure has definable Skolem functions, there is $r_0\in\dcl_{\mathcal{L}_{0} } (R_0 (c_i )_{i \in I} (d_j )_{j\in J})$ also contained in this set, and thus $(R_0,C)\models \exists x \varphi(x)$.
\end{proof}

\begin{theorem}\label{qe}
Fix a model $(R,C)$ where the cut $C$ has property $\ast$.  Choose a sequence $(c_i)_{i\in I}$ cofinal in $C$ and a sequence $(d_j)_{j\in J}$ coinitial in $R^{>C}$.  Let $R_0$ be $\dcl_\mathcal{L}((c_i)_{i\in I} (d_j)_{j\in J})$ and let $\mathcal{L}_{R_0}$ be the language $\mathcal{L}$ with constant symbols for the $c_i$ and $d_j$.  We let $T_{R_0}$ be the theory of $(R,C)$ in $\mathcal{L}_{R_0}$.  Then $T_{R_0}$ is universally axiomatizable, has quantifier elimination and definable Skolem functions.
\end{theorem}
\marginpar{}
\begin{proof}
By Lemma \ref{existential theory}, $(R_0, C)$ has property $\ast$.  So we may apply Lemma \ref{existential theory} to embed a copy of $(R,C)$ into an elementary extension $(\widetilde{R}, C)$ of $(R_0,C)$.  Then we apply Proposition \ref{substr} to see that $(R_0, C)$ is an elementary substructure of the copy of $(R,C)$, showing that the $\mathcal{L}_{R_0}$-theory of $(R_0,C)$ is $T_{R_0}$. 

By choosing a sufficiently saturated $(\widetilde{R}, C)$, universal axiomatizability follows from showing that every substructure of $(\widetilde{R}, C)$ is a submodel, and model completeness follows from showing every submodel of $(\widetilde{R}, C)$ is an elementary submodel.  Since any $\mathcal{L}_{R_0}$-substructure contains $R_0$, both of these follow from Proposition \ref{substr}.

A model complete theory eliminates quantifiers precisely when $T^{\forall}$ has the amalgamation property.  Thus, model completeness together with universal axiomatizability implies quantifier elimination. Furthermore, quantifier elimination together with uniform axiomatizability implies the existence of definable Skolem functions.
\end{proof}

\begin{corollary}
If $f\colon R\to R$ is $(R,C)$-definable, then there are $R$-definable functions $f_1 , \dots ,f_k \colon R\to R$ such that for each $a\in R$, $f(a)=f_i (a)$ for some $i\in \{ 1, \dots ,k\}$.  
\end{corollary}

\begin{proof}
This is a consequence of $T_A$ having definable Skolem functions and a universal axiomatization. 
\end{proof}

Finally, we have a partial converse to Theorem \ref{qe}.

\begin{proposition}
Suppose that the theory of $(R,C)$ has quantifier elimination and is universally axiomatizable. Then any sufficiently saturated elementary extension $(\widetilde{R},C)$ has property $\ast$.
\end{proposition}

\begin{proof}
 Consider $R\preceq \mathcal{R}$ but without choosing an interpretation of the predicate $C$ in $\mathcal{R}$.  Choose $a\in \mathcal{R}$ realizing the cut corresponding to $C(R)$ and define $C(\ra)$ to be the convex hull of $C(R)$ in $\ra$.  Repeat this process until one has built $(R_1, C)$ with the coinitiality  of $R_1^{>C(R_1)}$ greater than the cofinality of $C(R_1)$ (which stay equal to the cofinality of $C(R)$).  (Or if the  coinitiality  of $R^{>C(R)}$ was uncountable, repeat only countably many times to obtain $(R_1, C)$ with the coinitiality  of $R_1^{>C(R_1)}$ countable.)
 
 Since the cofinality of $C(R_1)$ and the cointiality of $R_1^{>C(R_1)}$ differ, Proposition \ref{star} implies that $(R_1, C)$ has property $\ast$.  $(R,C)$ is a substructure of $(R_1, C)$, and thus by quantifier elimination and universal axiomatizabilty, $(R,C)\preceq (R_1, C)$.  Moreover, $(R_1,C)$ embeds into any sufficiently saturated elementary extension $(\widetilde{R}, C)$ of $(R,C)$, and we may build up to $(\widetilde{R}, C)$ from $(R_1, C)$ with each step preserving property $\ast$.
\end{proof}

\end{section}

\begin{section}{Appendix -- convex valuations with o-minimal residue field}

Here we show that if $C$ is a predicate for a convex subring (hence valuation ring) of $R$ with o-minimal residue field, then the cut defined by $C$ has property $\ast$.

In this section, we assume that $C=V$ is realized as a proper convex subring of its ambient o-minimal field.
For an o-minimal field $S$ we denote by $\overline{x}$ the residue of $x \in V(S)$, and
we let $$I_S =\{x\in S\colon 0\leq x \leq 1\}.$$  For $X\subseteq S^{1+n}$, we set \[X(a):=\{ (\underline{x})\in S^n \colon (a, \underline{x}) \in X\},\] where $\underline{x}=(x_1 \dots ,x_n)$.  We assume that $p\in S_1 (\mathcal{R})$ is the global $R$-invariant type $$\{r<x\colon r\in V(R)\} \cup \{x<r\colon r\in \mathcal{R}^{>V(R)}\}.$$

\begin{definition} For an o-minimal field $S$, we say that 
$(S,V)\models \Sigma (n)$ if there is $\epsilon_0 \in \ma^{>0}$ such that for each $\epsilon \in \ma^{>\epsilon_0 }$,
\[ \overline{X(\epsilon_0 )}=\overline{X(\epsilon )} \] where $X\subseteq I_{S}^{1+n}$ is definable in $S$, and $\ma$ is the maximal ideal of $V(S)$.
\end{definition}
We shall use the following two facts. 
\begin{fact}[\cite{ax}, Theorem 1.2]\label{ominresfield}
The residue field with structure induced from $R$ (equivalently, from $(R,V)$ -- see \cite{mc}, Theorem 4.6) is o-minimal iff $(R,V)\models \Sigma (1)$.
\end{fact}

\begin{fact}[M., van den Dries \cite{jalou}, Theorem 1.2]\label{sigma1implesigman}
For any $n$, \[(R,V) \models \Sigma (1) \Rightarrow \Sigma (n) .\]
\end{fact}

\begin{lemma}\label{sigma}
Suppose $(R,V)\models \Sigma (1)$, $t$ realizes $p|_R$, and let $R_t =R\langle t \rangle$. Suppose that for each $a\in R_t$ with $V(R)<a<R^{>V(R)}$ one has $V(a)$. Then $(R_t ,V )\models \Sigma (1)$. 
\end{lemma}
\begin{proof}
We shall denote by $\ma$ and $\ma_t$ the maximal ideal of $V(R)$ and of $V(R_t)$ respectively, and note that $\ma_t$ is the convex hull of $\ma$ in $R_t$.
Assume to the contrary $(R_t ,V )\models \neg \Sigma (1)$.  Then there is an $R$-definable function $f\colon I_{R_t}^{2} \to I_{R_t}$ such that for all $\epsilon \in \ma_{t}$ there is $\delta_{\epsilon} \in \ma_{t}^{>\epsilon}$ with $$\overline{f (\frac{1}{t}, \epsilon )} \not=\overline{f (\frac{1}{t}, \delta_{\epsilon })}.$$
By Fact \ref{sigma1implesigman}, we can find $\epsilon_0 \in \ma^{>0}$ such that 
\begin{equation} \label{eq: sameresfiber}
\overline{\big( {\Gamma f \cap (I_{R}\times \{ \epsilon \} \times I_{R})} \big)}= \overline{ \big({\Gamma f \cap (I_{R}\times \{ \epsilon_0 \} \times I_{R})}\big)} \in V(R)/\ma
\end{equation}
for all $\epsilon \in \ma^{>\epsilon_0}$.  We shall show that then $\overline {f(\frac{1}{t},\epsilon_0 )}=\overline {f(\frac{1}{t},\epsilon )} \in V(R_t)/\ma_t$ for all $\epsilon \in \ma_{t}^{>\epsilon_0 }$.

Fix $\epsilon \in \ma_{t}^{>\epsilon_0}$ such that $\overline{f(\frac{1}{t},\epsilon_0 )}\not=\overline{f (\frac{1}{t},\epsilon )}$ (where both residues are elements of $V(R_t)/\ma_t$). Since $\ma_t$ is the convex hull of $\ma$ in $R_t$, we may assume that $\epsilon \in \ma$.
Since $(R,V)\models \Sigma (1)$, there are $a\in \ma^{>0}$ and $b\in I_{R}^{> \ma}$ such that 

\begin{equation} \label{eq: graphs}
\begin{split}
&\overline{\Gamma f \cap ([a,b]\times \{ \epsilon \} \times I_{R})} \mbox{ and } \overline{\Gamma f \cap ([a,b]\times \{ \epsilon_0 \} \times I_{R})}\\ 
&\mbox{ are graphs of functions } [0,\overline{b}]\to V(R)/\ma .
\end{split}
\end{equation}
Define $g(x)=|f(x,\epsilon ) - f(x, \epsilon_0 )|$ for $x\in [a,b] \subseteq R_{t}$.  By \ref{eq: sameresfiber} and \ref{eq: graphs}, $g(x) \in \ma$ for all $x\in [a,b]\subseteq R$. After possibly shrinking $[a,b] \subseteq R$ subject to the condition that $a<\frac{1}{t}<b$, we may moreover assume that $g$ is continuous and monotone on $(a,b)$, contradicting $g(\frac{1}{t})>\ma$.
\end{proof}

\begin{proposition} \label{Sigma implies star}
If $(R,V)\models \Sigma (1)$, then the cut corresponding to $V$ has property $\ast$. 
\end{proposition}

\begin{proof}
Assume to the contrary that $\underline{t}=t_1 , \dots , t_n$ is a finite Morley sequence in $p|_R$ and $f\colon R^{n+1} \to R$ is $\mathcal{L}_0$-definable over $R$ such that $f_{\underline{t}}$ takes a cofinal segment of $V(R)$ to a coinitial segment of the convex hull of $R^{>V(R)}$ in $\mathcal{R}$.  Note that we may replace $\underline{t}$ with any other tuple satisfying the same o-minimal type over $R$, and thus we may assume that, letting $R_{\underline{t}}=R\langle \underline{t} \rangle$, one has $V(a)$ for each $a\in R_{\underline{t}}$ with $V(R)<a<R^{>V(R)}$.  Thus by Lemma \ref{sigma}, $(R_{\underline{t}} , V ) \models \Sigma (1)$. 

After replacing $f_{\underline{t}}$ (restricted to a suitable interval) with $\frac{1}{f_{\underline{t}}^{-1}} \circ \frac{1}{x}$, we obtain a function that maps a cofinal segment of $\ma$ to a coinitial segment of $V(R)^{>\ma}$, where $\ma$ is the maximal ideal of $V(R)$.  The function $f_{\underline{t}}$ is definable over $R_{t}$, decreasing, and maps a cofinal segment of $\ma_{\underline{t}}$ (the maximal ideal of $V(R_{\underline{t}} )$) to a proper convex subset of $V(R_{\underline{t}} )^{>\ma_{\underline{t}} }$ and, after composition with the residue map, we obtain a function definable in $R_{\underline{t}}$ whose image is not eventually constant on $\ma_{\underline{t}}$, contradicting $(R_{\underline{t}} , V ) \models \Sigma (1)$.   
 
\end{proof}

\noindent We now turn to Corollary 3.4 \cite{mc} in which we claim that $(R,V)\models \Sigma (1)$ whenever $V$ is the convex hull of $\mathbb{Q}$.  Our proof of this now requires property $\ast$.

\begin{theorem}\label{Cor 3.4}   [Corollary 3.4 \cite{mc}]
Consider $(R,V)$ where  $V$ is the convex hull of $\mathbb{Q}$.  If $(R,V)$ has property $\ast$, then   $(R,V)\models \Sigma (1)$.
\end{theorem}

\begin{proof}
Work in a sufficiently saturated elementary extension $(\mathcal{R}, V)$.   Note that each cut in $\mathbb{Q}$ is realized in $\mathcal{R}$.  Take an element $a\in \mathcal{R}$ realizing such a cut.   If $\ra$ contains no element $\tilde a$ with $V(R)<\tilde a<R^{>V(R)}$ then note that in $(\ra, V)$, $V$ is again the convex hull of $\mathbb{Q}$.  If $\ra$ contains $\tilde a$ with $V(R)<\tilde a<R^{>V(R)}$ then there is an $R$-definable function, $f$, mapping the cut of $a$ to the cut in $R$ corresponding to $V$.  If $f(a)>V$, then note that $V(\ra)$ is again the convex hull of $\mathbb{Q}$ (by Theorem \ref{elextnew} and Lemma \ref{basestep}).  If $V(f(a))$, choose $b\in \mathcal{R}$ with $V(\mathcal{R})<b<R^{>V(R)}$, and replace $a$ with $f^{-1}(b)$.  Again, $V(\ra)$ is the convex hull of $\mathbb{Q}$.
 
 By Lemma \ref{elsubstr}, $(\ra, V)\preceq (\mathcal{R}, V)$.  Thus we may repeatedly add realizations of cuts in $\mathbb{Q}$ while keeping $V$ the convex hull of $\mathbb{Q}$ until we have built a model $(R_1,V)$ which is an elementary extension of $(R,V)$ and which has $\mathbb{R}$ as a residue field.  This residue field is necessarily o-minimal, and hence $(R_1,V)\models\Sigma (1) $, as does $(R,V)$.
\end{proof}

\end{section}

\end{document}